\documentclass{article}
\usepackage[english]{babel}

\usepackage[letterpaper,top=2cm,bottom=2cm,left=2cm,right=2cm,marginparwidth=1.75cm]{geometry}

\usepackage{calrsfs}
\usepackage{amssymb}
\usepackage{amsmath}
\usepackage{amsthm}
\usepackage{subcaption}
\usepackage{multirow}
\usepackage{graphicx}
\usepackage{amsfonts}
\usepackage{mathtools}
\usepackage{mathrsfs}
\usepackage{placeins}

\newtheorem{remark}{Remark}

\newtheorem{assumption}{Assumption}
\newtheorem{theorem}{Theorem}
\newtheorem{proposition}{Proposition}

\usepackage[mathscr]{euscript}
\usepackage[dvipsnames]{xcolor}
\usepackage[colorlinks=true, allcolors=OliveGreen]{hyperref}
\usepackage{filecontents}
\usepackage{tikz}
\usepackage{pgfplots}
\usetikzlibrary{external}
\usepackage{authblk}
\usepackage{todonotes}
\usepackage{array}
\newcolumntype{C}{>{\centering\arraybackslash}p{2.5cm}}

\newcommand{\averagel}{\{\!\!\{}
\newcommand{\averager}{\}\!\!\}}

\newcommand{\jumpl}{[\![}
\newcommand{\jumpr}{]\!]}

\newcommand{\partition}{\mathcal{T}_h}
\newcommand{\facesinternal}{\mathcal{F}^\mathrm{I}_h}
\newcommand{\faces}{\mathcal{F}_h}
\newcommand{\facesN}{\mathcal{F}_h^N}
\newcommand{\facesD}{\mathcal{F}_h^D}
\newcommand{\facesboundary}{\mathcal{F}^\mathrm{B}_h}

\newcommand{\Wh}{W_{h,p}^\mathrm{DG}}
\DeclareMathAlphabet{\mathcalligra}{T1}{calligra}{m}{n}

\graphicspath{{Images/}}

\title{Structure Preserving Polytopal Discontinuous Galerkin Methods for the Numerical Modeling of Neurodegenerative Diseases 
\footnote{\textbf{Funding}: PFA has been partially funded by the research grants PRIN2017 n. 201744KLJL funded by MUR and PRIN2020 n. 20204LN5N5 funded by MUR. PPFA has been partially supported by ICSC—Centro Nazionale di Ricerca in High Performance Computing, Big Data, and Quantum Computing funded by European Union—NextGenerationEU. FB is partially funded by “INdAM - GNCS Project”, codice CUP E53C22001930001. MC, FB and PFA are members of INdAM-GNCS.}}

\author[1]{Mattia Corti}

\affil[1]{MOX-Dipartimento di Matematica, Politecnico di Milano, Piazza Leonardo da Vinci 32, Milan, 20133, Italy}

\author[1]{Francesca Bonizzoni}

\author[1]{Paola F. Antonietti}

\begin{document}
\maketitle

\begin{abstract}
Many neurodegenerative diseases are connected to the spreading of misfolded prionic proteins. In this paper, we analyse the process of misfolding and spreading of both $\alpha$-synuclein and Amyloid-$\beta$, related to Parkinson's and Alzheimer's diseases, respectively. We introduce and analyze a positivity-preserving numerical method for the discretization of the Fisher-Kolmogorov equation, modelling accumulation and spreading of prionic proteins. The proposed approximation method is based on the discontinuous Galerkin method on polygonal and polyhedral grids for space discretization and on $\vartheta-$method time integration scheme. We prove the existence of the discrete solution and a convergence result where the Implicit Euler scheme is employed for time integration. We show that the proposed approach is structure-preserving, in the sense that it guaranteed that the discrete solution is non-negative, a feature that is of paramount importance in practical application. The numerical verification of our numerical model is performed both using a manufactured solution and considering wavefront propagation in two-dimensional polygonal grids. Next, we present a simulation of $\alpha$-synuclein spreading in a two-dimensional brain slice in the sagittal plane. The polygonal mesh for this simulation is agglomerated maintaining the distinction of white and grey matter, taking advantage of the flexibility of PolyDG methods in the mesh construction. Finally, we simulate the spreading of Amyloid-$\beta$ in a patient-specific setting by using a three-dimensional geometry reconstructed from magnetic resonance images and an initial condition reconstructed from positron emission tomography. Our numerical simulations confirm that the proposed method is able to capture the evolution of Parkinson's and Alzheimer's diseases.
\end{abstract} 

\section{Introduction}

A neurodegenerative disease is a process that causes the progressive death or function loss of neurons. Many different pathologies belong to this group and some of them are called proteinopathies because their aetiology involves misfolding and aggregation of prions into toxic and insoluble proteins \cite{walker_neurodegenerative_2015}. Typical examples of proteins that undergo this process are $\alpha$-Synuclein, related to Parkinson's disease \cite{stefanis-alpha-syn}, and Amyloid-$\beta$, whose aggregation is a triggering mechanism of Alzheimer's disease \cite{bloom2014amyloid}.
\par
Recently, the mathematical modelling of prion dynamics has been studied to elucidate how the physical processes at the basis of the agglomeration and diffusion processes can be related to complex brain structure and functioning. A mathematical description of the spreading at the macroscopic level can be a useful tool in clinical practice, where the use of positron emission tomography imaging (PET) is often considered too invasive and expensive \cite{vanoostveenImagingTechniquesAlzheimer2021}. Moreover, for some pathologies, like $\alpha$-sinucleopathies, there not exist satisfactory chemical ligands \cite{korat_alpha-synuclein_2021} that prevent diagnostic investigations, and this computed-assisted modelling is mandatory. 
\par
Concerning the numerical modelling of neurodegeneration, the most diffused mathematical description of this phenomenon is based on the Fisher-Kolmogorov (FK) equation (also known as Fisher-KPP) \cite{fisher-1937, kolmogorov-1937}. This model is a nonlinear diffusion-reaction equation that is particularly suited to describe biological species' evolution \cite{fornari_prion-like_2019, corti2023uncertainty}. Many different numerical methods have been proposed to compute the approximate solution of the FK equation, also in the context of brain neurodegeneration. For example, we recall Finite Element Methods (FEM) \cite{weickenmeierPhysicsbasedModelExplains2019, engwer_estimating_2021}, Boundary Elements Methods (BEM) \cite{gortsas_local_2022}, and Discontinuous Galerkin (DG) methods \cite{corti:FK}.
\par
In the context of modelling neurodegenerative disorders, the solution $c$ of the FK problem represents the (relative) concentration of misfolded proteins, which needs to be non-negative. It can be shown that in the continuous formulation, the solution of the FK equation has two equilibrium states: $c=1$ and $c=0$ \cite{salsa:EDP}. However, due to the unstable nature of the second equilibrium, at the discrete level, it is fundamental to construct a positivity-preserving numerical method to avoid numerical instabilities that lead to unphysical (negative) numerical solutions \cite{corti:FK}. For this reason, some works analyze the construction of suitable positivity-preserving methods both within the context of finite differences \cite{macias-diaz_explicit_2012} and DG \cite{bonizzoni_structure-preserving_2020} methods. The latter work uses a change of variable based on the exponential transformation to ensure positivity and entropy preservation at the discrete level.
\par
Starting from the high-order idea of \cite{bonizzoni_structure-preserving_2020} - limited to simplicial meshes - in this work we present and analyse a DG formulation on polygonal/polyhedral grids (PolyDG). 
The proposed approach presents several advantages and novelties: (i) The flexibility in the construction of the mesh, based on mesh agglomeration \cite{manuzzi:CNN}. This plays an important role, especially because of the complexity of the geometrical domain of the application at hand, i.e., the human brain; (ii) The freedom in the choice of discretization parameters, like the polynomial degree, which might locally change, from element to element \cite{antonietti_highorder_2021}. In the context of brain neurodegeneration, where the geometrical complexity of the domain is an issue, the use of elementwise approximation orders allows us to reduce the computational cost, without affecting the correctness of wavefront propagation; (iii) The use of higher-order time integration. Indeed, to the timescales of the brain neurodegeneration process (that typically means over decades), the use of low-order time integration methods is not convenient to catch the wave propagation correctly. For this reason, we adopt a second-order time integration strategy; (iv) Finally, we consider spatially varying and discontinuous physical parameters, which are fundamental to correctly describe the axonal diffusion of prionic proteins \cite{weickenmeierPhysicsbasedModelExplains2019, brennan_role_2023}.
\par
From the point of view of the analysis, we extend the proof of the existence of the numerical solution provided in~\cite{bonizzoni_structure-preserving_2020} for the implicit Euler method, to the generic $\vartheta$-method. The proof of existence is based on the use of the Leray-Shauder fixed point theorem and relies on the coercivity and continuity of the diffusion term. Even though the convergence of the fully discrete numerical solution to the analytical one is not theoretically proved, it is numerically demonstrated in the case $\vartheta=0.5$ (Crank-Nicolson (CN) scheme), with application to brain neurodegenerative diseases, and it is shown that it outperforms first-order advancing schemes.
\par
Concerning the application to the modelling of neurodegenerative disorders, the typical solution of the FK model is a wavefront propagating inside the brain geometry. For this reason, we analyze also the capabilities of our method in approximating wavefronts, providing also a comparison with the DG method proposed in \cite{corti:FK}, which is proven to suffer possible instabilities due to the fact it does not preserve the positivity when low order polynomial degrees are employed.
\par
The paper is organized as follows. In Section~\ref{sec:model}, we introduce the FK mathematical model and discuss its application to neurodegeneration. In Section~\ref{sec:discret}, we introduce the PolyDG space discretization and the time discretization using the $\vartheta$-method. Moreover, we show the coercivity and continuity of the variational forms in order to prove the existence of the discrete solution, and we discuss the extension of the convergence results of the fully discrete formulation. In Section~\ref{sec:numericalresults}, we present some convergence tests with a known exact solution and we discuss the accuracy of the proposed scheme in approximating travelling waves in a two-dimensional setting, making a comparison with the DG method of \cite{corti:FK}. Section~\ref{sec:numericalresultsbrain} is dedicated to the application of the proposed method to $\alpha$-Synuclein spreading in Parkinson's disease in a two-dimensional framework, employing agglomerated polygonal meshes, and Amyloid-$\beta$ in Alzheimer's disease in a three-dimensional patient-specific geometry, with initial conditions reconstructed from PET images. Finally, in Section~\ref{sec:conclusion}, we draw some conclusions and discuss future developments.

\section{The mathematical model}
\label{sec:model}

In this section, we present the FK equation to describe the reaction and diffusion of misfolded proteins. Given the final time $T>0$, the problem depends on the time $t\in(0,T]$ and space $\boldsymbol{x}\in\Omega\subset\mathbb{R}^d$ ($d=2,3$) variables. The unknown is the relative concentration of the misfolded protein $c=c(\boldsymbol{x},t)$, taking values in the interval $[0,1]$. A detailed derivation of the model can be found in \cite{weickenmeierPhysicsbasedModelExplains2019}. The problem in its strong formulation reads as follows: Find $c=c(\boldsymbol{x},t)$ such that:
\begin{equation}
 \begin{dcases}
     \dfrac{\partial c}{\partial t} =\nabla \cdot(\mathbf{D} \nabla\, c) + \alpha\,c(1-c) + f,
    & \mathrm{in}\ \Omega\times(0,T],
    \\[8pt]
    (\mathbf{D}\nabla c) \cdot \boldsymbol{n} = 0, & \mathrm{on}\;\Gamma_N\times(0,T],
    \\[8pt]
    c = c_\mathrm{D}, & \mathrm{on}\;\Gamma_D\times(0,T],
    \\[8pt]
    c(0)=c_0, & \mathrm{in}\;\Omega,
    \\[8pt]
\end{dcases}
\label{eq:fk_strong}
\end{equation}
where $\alpha=\alpha(\boldsymbol{x})$ is the reaction parameter, representing the local conversion rate of the proteins from the healthy to the misfolded state, modelling also on the clearance mechanisms \cite{ringstad_brain-wide_nodate,hornkjol_csf_2022}, and $\mathbf{D}=\mathbf{D}(\boldsymbol{x})\in\mathbb{R}^{d\times d}$ is the diffusion tensor, denoting the spreading of misfolded protein. Finally, $f=f(\boldsymbol{x},t)$ is the forcing term modelling the external addition of mass. Concerning the boundary conditions, we impose null flux at the boundary $\Gamma_N$ of the domain, while $c_\mathrm{D}$ fixes the value of concentration on a part of the boundary $\Gamma_D$, where $\{\Gamma_D,\Gamma_N\}$ form a partition of $\partial\Omega$, namely, $\Gamma_D \cup \Gamma_N = \partial \Omega$, $\Gamma_D \cap \Gamma_N = \emptyset$, and $|\Gamma_D|>0$. 
\par
Due to the physical meaning of the solution $c$, we aim to construct a positivity-preserving numerical scheme. Following \cite{bonizzoni_structure-preserving_2020}, we apply the exponential transformation $c=e^\lambda$, where $\lambda=\lambda(\boldsymbol{x},t)$ becomes the new unknown of the problem. As a result, we obtain the following strong formulation of the problem: Find $\lambda=\lambda(\boldsymbol{x},t)$ such that:

\begin{equation}
 \begin{dcases}
     \dfrac{\partial e^\lambda}{\partial t} =\nabla \cdot(e^\lambda \mathbf{D} \nabla\, \lambda) + \alpha\,e^\lambda(1-e^\lambda) + f,
    & \mathrm{in}\,\Omega\times(0,T],
    \\[8pt]
    (\mathbf{D}\nabla \lambda) \cdot \boldsymbol{n} = 0, & \mathrm{on}\;\Gamma_N\times(0,T],
    \\[8pt]
    \lambda = \lambda_\mathrm{D}, & \mathrm{on}\;\Gamma_D\times(0,T],
    \\[8pt]
    \lambda(0)=\lambda_0, & \mathrm{in}\;\Omega,
    \\[8pt]
\end{dcases}
\label{eq:lambda_fk_strong}
\end{equation}
The homogeneous Neumann boundary condition in problem \eqref{eq:fk_strong} reflects a homogeneous Neumann boundary condition also in problem \eqref{eq:lambda_fk_strong}. Concerning the initial condition and the Dirichlet boundary term we impose that $c_0=e^{\lambda_0}$ and $c_\mathrm{D}=e^{\lambda_\mathrm{D}}$, respectively. 

We make the following assumption on the data regularity.
\begin{assumption}[Data's regularity]
We assume the following regularity on the data appearing in \eqref{eq:fk_strong}: 
\begin{itemize}
    \item $\alpha\in L^\infty(\Omega)$; 
    \item $\boldsymbol{\mathrm{D}}\in L^\infty(\Omega,\mathbb{R}^{d\times d})$, and $\exists d_0,D_0>0\;\forall\boldsymbol{\xi}\in \mathbb{R}^d:\; d_0|\boldsymbol{\xi}|^2 \leq \boldsymbol{\xi}^\top\mathbf{D}\boldsymbol{\xi} \leq D_0|\boldsymbol{\xi}|^2 \quad \forall \boldsymbol{\xi}\in\mathbb{R}^d$;
    \item $f\in L^2((0,T],L^2(\Omega))$;
    \item $\lambda_\mathrm{D} \in L^2((0,T]; H^{1/2}(\Gamma_D))$.
\end{itemize} 
\end{assumption}

\section{Numerical discretization}
\label{sec:discret}

This section presents the discretization of the continuous problem \eqref{eq:lambda_fk_strong}, which is based on the polygonal discontinuous Galerkin method for the space discretization and the $\vartheta-$method for the time advancement.

\subsection{Discrete setting and preliminary estimates}

Let $\partition$ be a polytopic mesh partition of the domain $\Omega$, being the collection of disjoint polygonal/polyhedral elements $K$. For each element $K\in \partition$, $|K|$ denotes the Hausdorff measure of the element, and $h_K$ denotes its diameter. We set $h=\max_{K\in\partition} h_K$. 
Given two neighboring elements $K_1,\, K_2\in\partition$, their interface is defined as the intersection of their $(d-1)-$dimensional facets. In the case of $d=2$, the interface is a collection of line segments and the set of all of them is denoted with $\faces$. In the case $d=3$, the interface can be a generic polygon; for this reason, we introduce a decomposition of the polygon in planar triangles collected in the set $\faces$. Finally, we decompose $\faces$ into the union of interior faces ($\facesinternal$) and boundary faces ($\facesboundary$), i.e. $\faces = \facesinternal \cup \facesboundary$. Moreover, we assume that $\facesboundary$ can be further split according to the corresponding boundary condition: $\facesboundary = \facesD \cup \facesN$, where $\facesD$ and $\facesN$ are the boundary faces contained in $\Gamma_D$ and $\Gamma_N$, respectively. The last assumption implies that any $F\in\facesboundary$ is contained in either $\Gamma_D$ or $\Gamma_N$.
\par
\begin{assumption}[Mesh Regularity \cite{dipietro:HHO}]
The mesh sequence $\{\partition\}_h$ satisfies the following properties:
\begin{enumerate}
    \item Shape Regularity: $\forall K\in\partition\;it\;holds:\quad c_1 h_K^d\lesssim q|K|\lesssim  c_2h_K^d$.
    \item Contact Regularity: $\forall F\in\faces$ with $F\subseteq \overline{K}$ for some $K\in\partition$, it holds $h_K^{d-1}\lesssim |F|$, where $|F|$ is the Hausdorff measure of the face $F$.
    \item Submesh Condition: There exists a shape-regular, conforming, matching simplicial submesh $\widetilde{\partition}$ such that:
    \begin{itemize}
        \item $\forall \widetilde{K}\in\widetilde{\partition}\;\exists K\in\partition:\quad \widetilde{K}\subseteq K$.
        \item The family $\{\widetilde{\partition}\}_h$ is shape and contact regular.
        \item $\forall \widetilde{K}\in\widetilde{\partition}, K\in\partition$ with $\widetilde{K} \subseteq K$, it holds $h_K \lesssim h_{\widetilde{K}}$.
    \end{itemize}
\end{enumerate}
\end{assumption}
\begin{remark}
    We remark that most of the analysis is valid also under milder assumptions on the mesh \cite{cangianiVersionDiscontinuousGalerkin2022}; however in this work, we need to refer to the ones in Assumption 2. The technical point is the validity of \eqref{eq:in_tr_in} that holds under mesh assumptions of Assumption 2.3. However, we notice that from the numerical results of Sections 4 and 5, the assumption seems not to be needed in practice. 
\end{remark}
Concerning the space discretization, we introduce the following discontinuous finite element spaces with an elementwise variable polynomial degree:
\begin{equation*}
    \Wh = \{w\in L^2(\Omega):\quad w|_K\in\mathbb{P}_{p_K}(K)\quad\forall K\in\partition\},
\end{equation*}
\begin{equation*}
    \mathbf{W}_{h,p}^{\mathrm{DG}} = \{\mathrm{W}\in L^2(\Omega;\mathbb{R}^{d\times d}):\quad \mathrm{W}|_K\in\mathbb{P}_{p_K}^{d\times d}(K)\quad\forall K\in\partition\},
\end{equation*}
where $\mathbb{P}_{p_K}(K)$ is the space of polynomials of total degree $p_K\geq 1$ over a mesh element $K$. Concerning the physical data, we assume $\mathbf{D}\in\mathbf{W}_{h,p}^{\mathrm{DG}}$ and $\alpha\in\Wh$. We introduce the following trace operators \cite{arnoldUnifiedAnalysisDiscontinuous2001}. Let  $F\in\facesinternal$ be a face shared by the elements $K^\pm$ and let $\boldsymbol{n}^\pm$ be the unit normal vector on face $F$ pointing exterior to $K^\pm$, respectively. Then, for sufficiently regular scalar-valued functions $v$ and vector-valued functions $\boldsymbol{q}$, we define:
\begin{itemize}
    \item the average operator $\averagel{\cdot}\averager$ on $F\in \facesinternal$: $\averagel{v}\averager = \dfrac{1}{2} (v^+ + v^-), \quad \averagel{\boldsymbol{q}}\averager = \dfrac{1}{2} (\boldsymbol{q}^+ + \boldsymbol{q}^-)$;
    \item the jump operator $\jumpl{\cdot}\jumpr$ on $F\in \facesinternal$: $\jumpl{v}\jumpr = v^+\boldsymbol{n}^+ + v^-\boldsymbol{n}^-, \quad \jumpl{\boldsymbol{q}}\jumpr = \boldsymbol{q}^+\cdot\boldsymbol{n}^+ + \boldsymbol{q}^-\cdot\boldsymbol{n}$.
\end{itemize}
The superscripts $\pm$ denote the traces of the functions on $F$ taken within the interior to $K^\pm$. In an analogous way, on the face $F\in\facesD$ associated with the cell $K\in\partition$ with $\boldsymbol{n}$ outward unit normal on $\partial\Omega$, we define:
\begin{itemize}
    \item the average operator $\averagel{\cdot}\averager$ on $F\in\facesD$: $\averagel{v}\averager = v, \quad \averagel{\boldsymbol{q}}\averager = \boldsymbol{q}$;
    \item the standard jump operator $\jumpl{\cdot}\jumpr$ on $F\in\facesD$, with Dirichlet conditions $g$, $\boldsymbol{g}$: $\jumpl{v}\jumpr = (v-g)\boldsymbol{n}, \quad \jumpl{\boldsymbol{q}}\jumpr = (\boldsymbol{q}-\boldsymbol{g})\cdot\boldsymbol{n}$.
\end{itemize}
\par
Let us introduce the following broken Sobolev spaces for an integer $r\geq1$: $H^r(\mathcal{T}_h) = \{w_h\in L^2(\Omega): w_h|_K\in H^r(K)\quad \forall K\in\mathcal{T}_h\}$. Moreover, we introduce the shorthand notation for the $L^2$-norm $\|\cdot\|=\|\cdot\|_{L^2(\Omega)}$ and for the $L^2$-norm on a set of faces $\mathcal{F}$ as $\|\cdot\|_\mathcal{F}=\left(\sum_{F\in\mathcal{F}}\|\cdot\|_{L^2(F)}^2\right)^{1/2}$.
We define the following penalization function $\eta:\faces\rightarrow\mathbb{R}_+$:
\begin{equation}
    \eta(\lambda,p,h,D) = \max\{(e^\lambda)_{+},(e^\lambda)_{-}\}\max\left\{e^{\|\lambda\|_{L^\infty(K_{+})}},e^{\|\lambda\|_{L^\infty(K_{-})}}\right\} \zeta(p,h,D),
    \label{eq:penalty}
\end{equation}
where $\zeta(p,h,D)$ is defined as
\begin{equation}
    \zeta(p,h,D) = \eta_0
    \begin{cases}
         \{D_K\}_{\mathrm{A}}\dfrac{\{p_K^2\}_\mathrm{A}}{\{h_K\}_\mathrm{H}},  & \mathrm{on}\; F\in\facesinternal\\
         D_K\dfrac{p_K^2}{h_K},                 & \mathrm{on}\; F\in\facesD
    \end{cases}.
    \label{eq:penalty2}
\end{equation}
We point out that in Equation \eqref{eq:penalty}, we are considering both the harmonic average operator $\{\cdot\}_\mathrm{H}$ and the arithmetic average operator $\{\cdot\}_\mathrm{A}$ on $F\in\facesinternal$ and $\eta_0$ is a parameter at our disposal (to be chosen large enough to ensure stability). Moreover, we are defining $D_K = \|\sqrt{\mathbf{D}}|_K\|_2^2$. Finally, we can define the following DG-norm:
\begin{equation}
    \|c\|_{\mathrm{DG}}^2 = \left\|\sqrt{\mathbf{D}}\nabla_h c \right\|^2 + \|\sqrt{\zeta}\jumpl c\jumpr\|_{\facesinternal\cup\facesD}^2 \qquad \forall c\in H^1(\partition).
    \label{eq:DG-norm}
\end{equation}
\begin{remark}
    The choice of using this combination of harmonic and arithmetic averages is fundamental to obtaining the coercivity and continuity bounds of Propositions 1 and 2 below. 
\end{remark}
\par
Finally, we recall the result of inverse trace inequality \cite{riviere_priori_2002}:
\begin{equation}
    \exists C_I >0: \qquad \|v\|^2_{L^2(\partial K)} \leq C_I \dfrac{p^2_K}{h_K}\|v\|_{L^2(K)},\qquad \forall v\in\Wh,\;K\in\partition.
    \label{eq:in_tr_in}
\end{equation}

\subsection{PolyDG semi-discrete formulation}
To construct the semi-discrete formulation, we first introduce the interior penalty DG discretization of the nonlinear diffusion term $\mathcal{A}:\Wh\times \Wh\times \Wh\rightarrow \mathbb{R}$ as:
\begin{equation}
\begin{split}
    \mathcal{A}(u;v,w) = & \int_{\Omega} e^u \left(\mathbf{D}\nabla_h v\cdot\nabla_h w\right) -\sum_{F\in\facesinternal\cup\facesD}\int_{F}\left(\averagel e^u \mathbf{D} \nabla v\averager \cdot \jumpl w \jumpr +  \jumpl v\jumpr \cdot \averagel e^u \mathbf{D} \nabla w\averager\right)\mathrm{d}\sigma \\
    + & \sum_{F\in\facesinternal\cup\facesD}\int_{F}\eta(u) \jumpl v\jumpr \cdot \jumpl w\jumpr \mathrm{d}\sigma\qquad \forall u,v,w\in\Wh,
\end{split}
\label{eq:form_A_def}
\end{equation}
where $\nabla_h \cdot$ is the elementwise gradient \cite{quarteroni:EDP} and $\eta$ is defined as in Equation \eqref{eq:penalty}. The semi-discrete PolyDG formulation reads as follows:
\par
\bigskip
For any $t\in(0,T]$, find $\lambda_h(t)\in \Wh$ such that:
\begin{equation}
\begin{dcases}
     \left(\dfrac{\partial e^{\lambda_h(t)}}{\partial t},\varphi_h\right)_\Omega + \mathcal{A}(\lambda_h(t);\lambda_h(t),\varphi_h) - \left(\alpha e^{\lambda_h}\left(1-e^{\lambda_h}\right),\varphi_h\right)_\Omega = F(\varphi_h)
     & \forall \varphi_h\in \Wh, \\[8pt]
    \lambda_h(0)=\lambda_{0h}
\end{dcases}
\label{eq:DGFormulation}
\end{equation}
where $\lambda_{0h}\in\Wh$ is a suitable approximation of $\lambda_0\in W$. We next show some preliminary estimates that will be needed in the forthcoming well-posedness and convergence analysis.

\begin{proposition}[Coercivity of $\mathcal{A}$]
The form $\mathcal{A}$, defined in Equation \eqref{eq:form_A_def}, satisfies for all $v\in\Wh$:
\begin{equation}
    \mathcal{A}(v;v,v) \geq \dfrac{1}{2} \|e^{v/2}\|_\mathrm{DG}^2,
    \label{eq:ACoercivity}
\end{equation}
under the assumption on the penalty parameter value $\eta_0 \geq 16 C_I^2D_0$, where $d_0$ and $D_0$ are defined as in Assumption 1, and $C_I$ is the inverse trace inequality constant of relation \eqref{eq:in_tr_in}.
\end{proposition}
\begin{proof}
    Taking $u=v=w$ in Equation \eqref{eq:form_A_def}, we have:
    \begin{equation}
        \mathcal{A}(v;v,v) = \underset{\mathrm{(I)}}{\underbrace{\int_\Omega e^v(\mathbf{D}\nabla_h v) \cdot \nabla_h v}} - \underset{\mathrm{(II)}}{\underbrace{2 \sum_{F\in \facesinternal\cup\facesD}  \int_F  \averagel e^v\mathbf{D}\nabla v\averager\cdot\jumpl v\jumpr \mathrm{d}\sigma}} + \sum_{F\in \facesinternal\cup\facesD} \int_F \eta(v) |\jumpl v\jumpr|^2 \mathrm{d}\sigma.
    \end{equation}
    By treating each term separately, we obtain for $\mathrm{(I)}$ the following estimate:
    \begin{equation}
        \mathrm{(I)} \geq \int_\Omega  e^v|\sqrt{\mathbf{D}}\nabla_h v|^2 = \int_\Omega |\sqrt{\mathbf{D}}e^{v/2}\nabla_h v|^2 = 4 \int_\Omega |\sqrt{\mathbf{D}}\nabla_h e^{v/2}|^2.
    \end{equation}
    Then we control the term $\mathrm{(II)}$ by means of the Young's inequality:
    \begin{equation}
        |\mathrm{(II)}| \leq \underset{\mathrm{(III)}}{\underbrace{\sum_{F\in \facesinternal\cup\facesD}\int_F \beta_F |\averagel e^v\mathbf{D}\nabla v\averager|^2 \mathrm{d}\sigma}} +   \sum_{F\in \facesinternal\cup\facesD} \int_F \dfrac{1}{\beta_F} |\jumpl v\jumpr|^2 \mathrm{d}\sigma,
    \end{equation}
    where $\beta_F>0$ is a parameter we define as follows: 
    \begin{equation}
         \beta_F =  \dfrac{\min\left\{e^{-\|v\|_{L^\infty(K_{+})}},e^{-\|v\|_{L^\infty(K_{-})}}\right\}}{8 D_0 C_I^2\max\{(e^v)_{+},(e^v)_{-}\}}     \begin{cases}
         \dfrac{\{h_K\}_\mathrm{H}}{\{D_K\}_\mathrm{A}  \{p_K^2\}_\mathrm{A}},  & \mathrm{on}\; F\in\facesinternal,\\
         \dfrac{h_K}{D_K p_K^2},                 & \mathrm{on}\; F\in\facesD.
    \end{cases}
    \label{eq:beta_def}
    \end{equation}
    In \eqref{eq:beta_def} $d_0$ and $D_0$ are defined as in Assumption 1 and $C_I$ is the inverse trace inequality constant of relation \eqref{eq:in_tr_in}. Let us recall the following relation:
    \begin{equation}
        \dfrac{\{h_K\}_\mathrm{H}}{ \{D_K\}_\mathrm{A} \{p_K^2\}_\mathrm{A}} \leq 4 \min\left\{\dfrac{h_{K_-}}{D_{K_-} p^2_{K_-}},\dfrac{h_{K_+}}{D_{K_+} p^2_{K_+}}\right\}.
        \label{eq:meansrealtion}
    \end{equation}
    Then, by applying the inverse trace inequality and relation \eqref{eq:meansrealtion} we obtain:
    \begin{equation*}
        \begin{split}
        \mathrm{(III)} \leq & \sum_{K\in\partition} \dfrac{1}{8 D_0}\int_{\partial K} 4\dfrac{h_K\,e^{-\|v\|_{L^\infty(K)}}}{C_I^2 D_K p_K^2} |\mathbf{D} \nabla v|^2 \mathrm{d}\sigma \\
        \leq & \sum_{K\in\partition} \dfrac{1}{2}\int_{K} e^{-\|v\|_{L^\infty(K)}} |\sqrt{\mathbf{D}}\nabla v|^2 \leq \dfrac{1}{2}\int_{\Omega} e^v |\sqrt{\mathbf{D}}\nabla_h v|^2= 2\int_{\Omega} |\sqrt{\mathbf{D}}\nabla_h e^{v/2}|^2.
        \end{split}
    \end{equation*}
    Inserting the above estimates in Equation \eqref{eq:ACoercivity}, we obtain:
    \begin{equation}
        \mathcal{A}(v;v,v) \geq 2\int_{\Omega} |\sqrt{\mathbf{D}}\nabla_h e^{v/2}|^2 + \sum_{F\in \facesinternal\cup\facesD} \int_F \left(\eta(v)-\dfrac{1}{\beta_F}\right) |\jumpl v\jumpr|^2 \mathrm{d}\sigma.
        \label{eq:boundAvvv}
    \end{equation}
    For $F\in\facesinternal$, the second integral on the rhs of Equation \eqref{eq:boundAvvv} is positive provided that:
\begin{equation*}
    \eta(v)-\dfrac{1}{\beta_F} = \left(\eta_0   - 8C_I^2D_0\right)\dfrac{\{D_K\}_{\mathrm{A}}\{p_K^2\}_\mathrm{A}}{\{h_K\}_\mathrm{H}}\max\{(e^v)_{+},(e^v)_{-}\}\max\left\{e^{\|v\|_{L^\infty(K_{+})}},e^{\|v\|_{L^\infty(K_{-})}}\right\}>0.
\end{equation*}
The same bound can be obtained on $F\in\facesD$. 
By taking $\eta_0 \geq 16C_I^2D_0$ the positivity is guaranteed, and by exploiting the following relation 
$$\max\{(e^v)_{+},(e^v)_{-}\}\max\left\{e^{\|v\|_{L^\infty(K_{+})}},e^{\|v\|_{L^\infty(K_{-})}}\right\}\geq 1
$$  
we obtain:
\begin{equation}
\begin{split}
    \mathcal{A}(v;v,v) \geq & 2\int_{\Omega} |\sqrt{\mathbf{D}}\nabla_h e^{v/2}|^2 + \sum_{F\in \facesinternal} \dfrac{\eta_0}{2} \int_F \{D_K\}_{\mathrm{A}}\dfrac{\{p_K^2\}_\mathrm{A}}{\{h_K\}_\mathrm{H}} |\jumpl e^{v/2}\jumpr|^2 \mathrm{d}\sigma +  \sum_{F\in \facesD} \dfrac{\eta_0}{2} \int_F D_K\dfrac{p_K^2}{h_K} |e^{v/2}|^2 \mathrm{d}\sigma = \\
    = & 2\int_{\Omega} |\sqrt{\mathbf{D}}\nabla_h e^{v/2}|^2 +   \sum_{F\in \facesinternal\cup \facesD} \dfrac{1}{2} \int_F \zeta |\jumpl e^{v/2}\jumpr|^2 \mathrm{d}\sigma \geq \dfrac{1}{2} \|e^{v/2}\|_\mathrm{DG}^2,
    \end{split}
\end{equation}
where $\zeta$ has been defined in~\eqref{eq:penalty2}.
\end{proof}
\begin{proposition}[Continuity of $\mathcal{A}$]
The form $\mathcal{A}$, defined in Equation \eqref{eq:form_A_def}, satisfies for all $u, v\in\Wh$:
\begin{equation}
    \left|\mathcal{A}(u;u,v)\right| \leq \mu \max_{K\in\partition} \{e^{\|u\|_{L^\infty(K)}}\} \| e^u\|_{\mathrm{DG}} \,\| u\|_{\mathrm{DG}} \,\| v\|_{\mathrm{DG}} \, ,
    \label{eq:AContinuity}
\end{equation}
with $\mu:= \max\left\{1,\sqrt{\frac{4D_0C_I^2}{d_0\eta_0}}\right\}$, where $d_0$ and $D_0$ are defined as in Assumption 1, $C_I$ is the inverse trace inequality constant in \eqref{eq:in_tr_in}, and $\eta_0$ is the penalty constant introduced in~\ref{eq:penalty2}.
\end{proposition}
\begin{proof}
   From Equation \eqref{eq:form_A_def} we obtain:
    \begin{equation}
        \mathcal{A}(u;u,v) = \underset{\mathrm{(I)}}{\underbrace{\int_\Omega e^u(\mathbf{D}\nabla_h u) \cdot \nabla_h v}} + \sum_{F\in \facesinternal\cup\facesD} 
         \underset{\mathrm{(II)}}{\underbrace{\int_F \eta(u) \jumpl u\jumpr\cdot \jumpl v\jumpr \mathrm{d}\sigma}} - \underset{\mathrm{(III)}}{\underbrace{\int_F  \averagel e^u\mathbf{D}\nabla u\averager\cdot\jumpl v\jumpr \mathrm{d}\sigma}} -\underset{\mathrm{(IV)}}{\underbrace{ \int_F  \averagel e^u\mathbf{D}\nabla v\averager\cdot\jumpl u\jumpr \mathrm{d}\sigma}}
    \end{equation}
    By treating each term separately, we obtain for $\mathrm{(I)}$ the following estimate using the regularity assumption on $\mathrm{D}$ in Assumption 1:
    \begin{equation}
        |\mathrm{(I)}| \leq \int_\Omega  e^u|\sqrt{\mathbf{D}}\nabla_h u|\ |\sqrt{\mathbf{D}}\nabla_h v| = \int_\Omega |\sqrt{\mathbf{D}}\nabla_h e^u|\ |\sqrt{\mathbf{D}}\nabla_h v|  \leq \|\sqrt{\mathbf{D}}\nabla_h e^u\|\,\|\sqrt{\mathbf{D}}\nabla_h v\|.
    \end{equation}
    Then, we control the term $\mathrm{(II)}$ by means of the Young's inequality:
    \begin{equation}
    |\mathrm{(II)}| \leq \max_{K\in\partition} \{e^{\|u\|_{L^\infty(K)}} \} \|\sqrt{\zeta} \jumpl u\jumpr\|_{\facesinternal\cup\facesD}\|\sqrt{\zeta} \jumpl v\jumpr\|_{\facesinternal\cup\facesD}.
    \end{equation}
    The bound on term $\mathrm{(III)}$ follows thanks to the Young's inequality:
    \begin{equation}
    |\mathrm{(III)}| \leq \sum_{F\in \facesinternal\cup\facesD}  \Big(\underset{\mathrm{(V)}}{\underbrace{\int_F \gamma_F |\averagel e^u\mathbf{D}\nabla u\averager|^2 \mathrm{d}\sigma}}\Big)^{1/2}\Big(\int_F \dfrac{1}{\gamma_F} |\jumpl v\jumpr|^2 \mathrm{d}\sigma\Big)^{1/2},
    \end{equation}
    where $\gamma_F>0$ is defined as follows: 
    \begin{equation}
         \gamma_F =  \dfrac{d_0^2}{8 D_0 C_I^2}     \begin{cases}
         \dfrac{\{h_K\}_\mathrm{H}}{\{D_K\}_\mathrm{A}  \{p_K^2\}_\mathrm{A}},  & \mathrm{on}\; F\in\facesinternal,\\
         \dfrac{h_K}{D_K p_K^2},                 & \mathrm{on}\; F\in\facesD.
    \end{cases}
    \end{equation}
    Let us recall the following relation:
    \begin{equation}
        \dfrac{\{h_K\}_\mathrm{H}}{ \{D_K\}_\mathrm{A} \{p_K^2\}_\mathrm{A}} \leq 4 \min\left\{\dfrac{h_{K_-}}{D_{K_-} p^2_{K_-}},\dfrac{h_{K_+}}{D_{K_+} p^2_{K_+}}\right\}.
        \label{eq:meanbounds}
    \end{equation}
    Then, by applying the inverse trace inequality in relation \eqref{eq:in_tr_in} and relation \eqref{eq:meanbounds} we obtain:
    \begin{equation*}
        |\mathrm{(V)}| \leq \sum_{K\in\partition} \dfrac{d_0^2}{8 D_0}\int_{\partial K} 4\dfrac{h_K}{C_I^2 D_K p_K^2} |e^u \mathbf{D} \nabla u|^2 \mathrm{d}\sigma \leq \sum_{K\in\partition} \dfrac{d_0}{2}\int_{K} |\sqrt{\mathbf{D}}\nabla e^u|^2 = \dfrac{d_0}{2} \|\sqrt{\mathbf{D}}\nabla_h e^u\|^2.
    \end{equation*}
    From the above estimates it follows:
    \begin{equation*}
        |\mathrm{(III)}| \leq \sqrt{\dfrac{4D_0C_I^2}{d_0\eta_0}} \|\sqrt{\mathbf{D}}\nabla_h e^u\| \,\|\sqrt{\zeta}\jumpl v\jumpr\|^2_{\facesinternal\cup\facesD}.
    \end{equation*}
    Finally, we estimate the term $\mathrm{(IV)}$ by applying the inverse trace inequality in relation \eqref{eq:in_tr_in} and relation \eqref{eq:meanbounds}:
    \begin{equation*}
        \begin{split}
        |\mathrm{(IV)}| \leq & \sum_{F\in \facesinternal\cup\facesD}  \Big(\int_F \gamma_F |\averagel e^u\mathbf{D}\nabla v\averager|^2 \mathrm{d}\sigma\Big)^{1/2} \Big(\int_F \dfrac{1}{\gamma_F} |\jumpl u\jumpr|^2 \mathrm{d}\sigma\Big)^{1/2} \\ \leq & \Big(\sum_{K\in\partition} \dfrac{d_0^2}{8 D_0}\int_{\partial K} 4\dfrac{h_K}{C_I^2 D_K p_K^2} |e^u \mathbf{D} \nabla v|^2 \mathrm{d}\sigma\Big)^{1/2} \sqrt{\dfrac{8D_0C_I^2}{d_0^2\eta_0}}\|\sqrt{\zeta}\jumpl u\jumpr\|_{\facesinternal\cup\facesD} \\
        \leq & \sqrt{\max_{K\in\partition} \{e^{\|u\|_{L^\infty(K)}}\}\dfrac{4D_0C_I^2}{d_0\eta_0}}\|\sqrt{\mathbf{D}}\nabla_h v\| \,\|\sqrt{\zeta}\jumpl u\jumpr\|_{\facesinternal\cup\facesD}.
        \end{split}
    \end{equation*}
    Finally, putting together all the previous bounds, we obtain:
    \begin{equation}
        \left|\mathcal{A}(u;u,v)\right|\leq  \max\left\{1,\sqrt{\dfrac{4D_0C_I^2}{d_0\eta_0}}\right\} \max_{K\in\partition} \{e^{\|u\|_{L^\infty(K)}}\} \| e^u\|_{\mathrm{DG}} \,\| u\|_{\mathrm{DG}} \,\| v\|_{\mathrm{DG}} \,,
    \end{equation}
    and the proof is complete.
\end{proof}

\subsection{Fully discrete formulation}
To discretize Equation \eqref{eq:DGFormulation} in time, we consider the $\vartheta-$method scheme. We remark that due to the nonlinear nature of the strong formulation with the change of variable, we need a nonlinear solver, and therefore using an implicit scheme for time integration does not affect the computational cost. In this section, we consider homogeneous Dirichlet conditions $\lambda_D=0$ for simplicity in the calculations. However, the results can be extended to the non-homogeneous case with proper regularity assumptions on $\lambda_D$.
\par
Let $\{t_\ell\}_{\ell=0}^{N_t}$ be the uniform partition of the time interval $[0,T]$ into $N_t$ intervals with length $dt=\frac{T}{N_t}$, namely, $0=t_0<t_1<...<t_{N_t}=T$ and $t_\ell=\frac{\ell T}{N_t}$ for $\ell=0,...,N_t$. Let us introduce a parameter $\varepsilon>0$. Then, the fully discrete formulation of problem \eqref{eq:DGFormulation} reads: given the initial condition $\lambda^0=\lambda_0$, find $\lambda^{k+1}_h$ for $k=0,...,N_t-1$, such that:
\begin{equation}
 \begin{split}
 \Bigg(\dfrac{e^{\lambda_h^{k+1}}-e^{\lambda_h^{k}}}{\Delta t},\varphi_h\Bigg)_\Omega - & \left(\alpha \left(\vartheta e^{\lambda_h^{k+1}}+(1-\vartheta) e^{\lambda_h^{k}}\right) \left(1-\left(\vartheta e^{\lambda_h^{k+1}}+(1-\vartheta) e^{\lambda_h^{k}}\right)\right),\varphi_h\right)_\Omega \\ + & \dfrac{\varepsilon}{\Delta t} (\lambda_h^{k+1},\varphi_h)_\Omega + \dfrac{\varepsilon}{\Delta t} (\mathbf{D}\nabla_h \lambda_h^{k+1},\nabla_h \varphi_h)_\Omega + \dfrac{\varepsilon}{\Delta t} (\zeta \jumpl \lambda_h^{k+1}\jumpr, \jumpl\varphi_h\jumpr)_{\facesinternal\cup\facesD} \\ + &  \vartheta\mathcal{A}(\lambda_h^{k+1};\lambda_h^{k+1},\varphi_h) + (1-\vartheta)\mathcal{A}(\lambda_h^k;\lambda_h^k,\varphi_h) =   \vartheta F^{k+1}(\varphi_h) + (1-\vartheta) F^{k}(\varphi_h), \qquad \mathrm{in}\;\Omega.
 \end{split}
\label{eq:lambda_fk_fully}
\end{equation}
The introduction of the additional regularizing terms proportional to the parameter $\varepsilon>0$ is fundamental to prove the existence of the solution via the Leray-Schauder fixed-point theorem \cite{bonizzoni_structure-preserving_2020}. However, from a numerical point of view, the presence of a $\varepsilon > 0$ is not really needed and it can be chosen equal to 0 in the simulations (see Section \ref{sec:numericalresults}).

We next prove that formulation \eqref{eq:lambda_fk_fully} admits a solution.
\begin{proposition}[Existence of a solution]
Let $\varepsilon>0$. Given $\lambda_h^{k}\in\Wh$, then the fully discrete formulation in Equation \eqref{eq:lambda_fk_fully} admits a solution $\lambda_h^k\in\Wh$, provided that Assumptions 1 and 2 hold and the penalty constant $\eta_0$ defined as in \eqref{eq:penalty2} is chosen sufficiently large.
\end{proposition}
\begin{proof}
    The proof is based on the application of the Leray-Schauder theorem. For clarity, we subdivide the proof into 3 steps.
    \subparagraph*{Step 1: Definition of the operator $\Phi$.}
    First of all, let us introduce the fixed point operator $\Phi:\Wh\times[0,1]\rightarrow \Wh$ such that $\Phi(w,\sigma)=v$ with $v\in\Wh$ being the unique solution of the linear problem:
    \begin{equation}
     \begin{split}
     \varepsilon (v,\phi)_\Omega + & \varepsilon (\mathbf{D}\nabla_h v,\nabla_h \phi)_\Omega + \varepsilon (\zeta\jumpl v\jumpr, \jumpl\phi\jumpr)_{\facesinternal\cup\facesD} = \sigma(e^{\lambda_h^{k}}-e^w,\phi)_\Omega \\ + &  \sigma\left(\alpha\Delta t(\vartheta e^w+(1-\vartheta) e^{\lambda_h^{k}}) (1-(\vartheta e^w+(1-\vartheta) e^{\lambda_h^{k}})),\phi\right)_\Omega  \\ - & \sigma\vartheta\Delta t\mathcal{A}(w;w,\phi) - \sigma(1-\vartheta)\Delta t\mathcal{A}(\lambda_h^k;\lambda_h^k,\phi) +  \sigma\vartheta\Delta t F^{k+1}(\phi) + \sigma(1-\vartheta)\Delta t F^{k}(\phi)\qquad\forall\phi\in\Wh.
     \end{split}
     \label{eq:fpprob}
    \end{equation}
    \subparagraph*{Step 2: Compactness of $\Phi$.}
    $\Phi$ is well defined by the Lax-Milgram lemma, thanks to the coercivity and continuity on $\Wh$ of the left-hand side of \eqref{eq:fpprob} and to the continuity of the right-hand side of \eqref{eq:fpprob}. Finally, we observe that $\Phi(w,0)=0$. Due to the finite dimension of the space $\Wh$, these properties are enough to prove also the compactness of the operator. 
    \subparagraph*{Step 3: Uniform bound for all the fixed points.} To prove the property of uniform bound we take $v\in\Wh$ and $\sigma\in[0,1]$ such that $v=\Phi(v,\sigma)$. First of all, let us notice that we can bound the right-hand side of \eqref{eq:fpprob} by using the coercivity of $\mathcal{A}$ and the existence of a constant $M=M(\lambda_h^{k})$ such that $\alpha\Delta t(\vartheta e^v+(1-\vartheta) e^{\lambda_h^{k}}) (1-\vartheta e^v)v \leq M(\lambda_h^{k})$. Indeed, there holds
    \begin{equation}
    \label{eq:DGbound}
    \begin{split}
    \varepsilon \|v\|^2 + \varepsilon \| v\|^2_{\mathrm{DG}} = & \sigma(e^{\lambda_h^{k}}-e^v,v)_\Omega +  \sigma\left(\alpha\Delta t(\vartheta e^v+(1-\vartheta) e^{\lambda_h^{k}}) (1-(\vartheta e^v+(1-\vartheta) e^{\lambda_h^{k}})),v\right)_\Omega  \\ & - \sigma\vartheta\Delta t\mathcal{A}(v;v,v) - \sigma(1-\vartheta)\Delta t\mathcal{A}(\lambda_h^k;\lambda_h^k,v) +  \sigma\vartheta\Delta t F^{k+1}(v) + \sigma(1-\vartheta)\Delta t F^{k}(v) \\
     \leq & \sigma(e^{\lambda_h^{k}}-e^v,v)_\Omega - \sigma\left(\alpha\Delta t (1-\vartheta) e^{\lambda_h^{k}}(\vartheta e^v+(1-\vartheta) e^{\lambda_h^{k}}),v\right)_\Omega + \sigma M(\lambda_h^{k})\\ & - \sigma(1-\vartheta)\Delta t\mathcal{A}(\lambda_h^k;\lambda_h^k,v) +  \sigma\vartheta\Delta t F^{k+1}(v) + \sigma(1-\vartheta)\Delta t F^{k}(v).
     \end{split}
    \end{equation}
    Then, by introducing the function $s(x)=x(\log(x)-1)+1\geq 0$ and exploiting its convexity we obtain:
    \begin{equation}
        (e^{\lambda_h^{k}}-e^v)v = (e^{\lambda_h^{k}}-e^v)s'(e^v) \leq s(e^{\lambda_h^{k}})-s(e^v).
    \end{equation}
    Thus, using also the fact that $-s(e^v)\leq 0$ and relation \eqref{eq:DGbound} we obtain:
    \begin{equation*}
     \begin{split}
     \varepsilon \|v\|^2 + \varepsilon\| v\|^2_{\mathrm{DG}} \leq & \sigma\int_\Omega \left(s(e^{\lambda_h^{k}}) (1 + \alpha\Delta t \vartheta(1-\vartheta) e^{\lambda_h^{k}}) \right) - \sigma\left(\alpha\Delta t (1-\vartheta)^2 e^{2\lambda_h^{k}},v\right)_\Omega + \sigma M(\lambda_h^{k})\\ - & \sigma(1-\vartheta)\Delta t\mathcal{A}(\lambda_h^k;\lambda_h^k,v) +  \sigma\vartheta\Delta t F^{k+1}(v) + \sigma(1-\vartheta)\Delta t F^{k}(v). 
     \end{split}
    \end{equation*}
    Using Equation \eqref{eq:AContinuity} and the Young's inequality with suitable coefficients $\epsilon_1$ and $\epsilon_2$, we get:
    \begin{equation}
    \begin{split}
        \left(\varepsilon-\dfrac{3}{2}\sigma\Delta t \epsilon_2\right) \|v\|^2 + & \left(\varepsilon-\dfrac{\sigma(1-\vartheta)\Delta t \mu \epsilon_1}{2}\right)\| v\|^2_{\mathrm{DG}} \leq  \sigma\int_\Omega \left(s(e^{\lambda_h^{k}}) (1 + \alpha\Delta t \vartheta(1-\vartheta) e^{\lambda_h^{k}}) \right) \\ - & \sigma\alpha\Delta t (1-\vartheta) \|e^{2\lambda_h^{k}}\|^2 + \sigma M(\lambda_h^{k}) + \dfrac{\sigma(1-\vartheta)\Delta t}{2\epsilon_1} \mu \max_{K\in\partition} \{e^{\|\lambda_h^k\|_{L^\infty(K)}}\}^2 \| e^{\lambda_h^k}\|_{\mathrm{DG}}^2 \,\| \lambda_h^k\|_{\mathrm{DG}}^2 \\ + &  \dfrac{\sigma\Delta t}{2\epsilon_2} \left(\vartheta \|f(t^{k+1})\|^2 + (1-\vartheta)\|f(t^{k+1})\|^2 \right).
    \end{split}
    \end{equation}
    By applying the Leray-Schauder theorem \cite{salsa:EDP} we derive the existence of a solution for problem \eqref{eq:lambda_fk_fully}, and the proof is complete.
\end{proof}

\subsection{Convergence of the discrete solution}
In this section, we prove the convergence of the solution to the PolyDG fully discrete formulation in Equation \eqref{eq:lambda_fk_fully} with $\vartheta=1$ (Implicit Euler method) to the solution of the continuous problem. An additional assumption we make in this proof is the forcing-free model $f=0$. The result follows by extending the convergence theorem proved in \cite{bonizzoni_structure-preserving_2020} to the case pf polytopal/polyhedral meshes and high-order approximations. 

Let us start introducing the notion of entropy $S:[0,T]\rightarrow\mathbb{R}$ of the system \cite{jungel:entropy}, namely
\begin{equation}
    S(t) = \int_\Omega \left(u(t)(\log(u(t))-1)+1\right).
\end{equation}
To prove the convergence of the numerical solution, we need to show that the discrete entropy $S_h^{k} = \int_\Omega (e^{\lambda_h^k}({\lambda_h^k}-1)+1)$ decays as $k\rightarrow\infty$ \cite{bonizzoni_structure-preserving_2020}, and that the DG norm (see equation~\eqref{eq:DG-norm}) of the discrete solution is uniformly bounded.
\begin{remark}
The analysis in this section is performed only for the case $\vartheta=1$. The treatment of the general case $\vartheta\in[0,1]$ is not straightforward, due to the presence of the components from the previous timestep that cannot be easily treated and prevent to recover the decay of the discrete entropy. Nevertheless, as it will be demonstrated in the numerical result sections, the scheme exhibits optimal convergence rates for any $\vartheta\in[0,1]$. The extension of the analysis to the case $\vartheta\neq 1$ is under investigation and will be the subject of future research.
\end{remark}
\begin{proposition}
    Let Assumptions 1 and 2 hold and let $\eta_0$ defined as in Equation \eqref{eq:penalty2} be chosen sufficiently large. Let $\lambda_h^{k+1}$ be the solution to problem \eqref{eq:lambda_fk_fully} with $\vartheta=1$, $\varepsilon>0$ and homogeneous forcing term $f=0$. Then:
    \begin{equation}
       \left\|e^{\lambda^{k+1}_h/2}\right\|_\mathrm{DG}^2 \leq \dfrac{ 2 S_h^0}{\Delta t},
    \end{equation}
    where $S_h^0$ is the initial discrete entropy.
\end{proposition}
\begin{proof}
    Let us consider the problem \eqref{eq:lambda_fk_fully} with $\vartheta=1$ and $\varphi_h = \lambda_h^{k+1}$:
    \begin{equation*}
     \Delta t\left(\alpha e^{\lambda_h^{k+1}} \left(e^{\lambda_h^{k+1}}-1\right),\lambda_h^{k+1}\right)_\Omega + \varepsilon \|\lambda_h^{k+1}\|^2 + \varepsilon \|\lambda_h^{k+1}\|_\mathrm{DG}^2 +  \Delta t\mathcal{A}(\lambda_h^{k+1};\lambda_h^{k+1},\lambda_h^{k+1}) = \left(e^{\lambda_h^{k}}-e^{\lambda_h^{k+1}},\lambda_h^{k+1}\right)_\Omega.
    \end{equation*}
    By observing that $e^v \left(e^v-1\right)v\geq 0$ for each $v\in\Wh$ and by using Equation \eqref{eq:ACoercivity}, we obtain:
    \begin{equation*}
    \frac{\Delta t}{2} \|e^{\lambda_h^{k+1}/2}\|_\mathrm{DG}^2 \leq \left(e^{\lambda_h^{k}}-e^{\lambda_h^{k+1}},\lambda_h^{k+1}\right)_\Omega.
    \end{equation*}
    Exploiting the convexity of the density of entropy function $s(v) = v(\log(v)-1)+1$ and noticing that $v = s'(v)$, we obtain:
    \begin{equation}
    \frac{\Delta t}{2} \|e^{\lambda_h^{k+1}/2}\|_\mathrm{DG}^2 \leq S_h^k - S_h^{k+1} \leq S_h^k \leq S_h^0,
    \label{eq:finalstepDG}
    \end{equation}
    where in the last step we used the discrete entropy inequality (\cite{bonizzoni_structure-preserving_2020}, Lemma 7). From Equation \eqref{eq:finalstepDG}, the thesis follows.
\end{proof}
\begin{theorem}[Convergence]
   Let Assumptions 1 and 2 hold and let $\eta_0$ be sufficiently large. Let $\varepsilon> 0$, $\vartheta=1$, $\Delta t \in (0,1)$, and let $\lambda_h^{k+1}\in\Wh$ be a solution to \eqref{eq:lambda_fk_fully} with homogeneous forcing term $f=0$. Assume that $\lambda_h^k\in\Wh$ is such that $e^{\lambda_h^k}\rightarrow c^{k}$ strongly in $L^2(\Omega)$ as $(\varepsilon,h)\rightarrow 0$. Then there exists a unique strong solution $c^{k+1}\in H^2(\Omega)$ to:
    \begin{equation}
    \begin{dcases}
    \dfrac{c^{k+1}-c^{k}}{\Delta t} =\nabla \cdot(\mathbf{D} \nabla\, c^{k+1}) + \alpha\,c^{k+1}(1-c^{k+1}), & \mathrm{in}\,\Omega,\\
    c^{k+1} = c_\mathrm{D} = e^{\lambda_\mathrm{D}}, & \mathrm{on}\,\Gamma_D,\\
    (\mathbf{D} \nabla c^{k+1})\cdot\boldsymbol{n} = 0, & \mathrm{on}\,\Gamma_N,\\
    \end{dcases}
\end{equation}
such that $e^{\lambda_h^{k+1}}\rightarrow c^{k+1}$ strongly in $L^2(\Omega)$ as $(\varepsilon,h)\rightarrow 0$.
\end{theorem}
The proof follows the same steps as in \cite{bonizzoni_structure-preserving_2020} and it makes use of Propositions 1 and 4, as well as of the extensions of variational inequalities valid for polygonal/polyhedral meshes.

\section{Numerical results: verification}
\label{sec:numericalresults}
In this section, we aim at verifying the accuracy of the method presented in section~\ref{sec:discret}.

\subsection{Test case 1: convergence analysis in two dimensions}
\begin{figure}[t!]
    \begin{subfigure}[b]{0.5\textwidth}
          \resizebox{\textwidth}{!}{\definecolor{mycolor2}{rgb}{0.00000,1.00000,1.00000}%
\pgfplotsset{
  log x ticks with fixed point/.style={
      xticklabel={
        \pgfkeys{/pgf/fpu=true}
        \pgfmathparse{exp(\tick)}%
        \pgfmathprintnumber[fixed  zerofill, precision=2]{\pgfmathresult}
        \pgfkeys{/pgf/fpu=false}
      }
  }
}
\begin{tikzpicture}

\begin{axis}[%
width=3.875in,
height=2.36in,
at={(2.6in,1.099in)},
scale only axis,
xmode=log,
xmin=0.064,
xmax=0.3239,
xminorticks=true,
xlabel = {$h$ [-]},
ylabel = {$||e^{\lambda(T)}-e^{\lambda_h(T)}||_{\mathrm{DG}}$},
ymode=log,
ymin=1e-9,
ymax=0.2,
yminorticks=true,
axis background/.style={fill=white},
title style={font=\bfseries},
xmajorgrids,
xminorgrids,
ymajorgrids,
yminorgrids,
legend style={legend cell align=left, align=left, draw=white!15!black}
]
              
\addplot [color=red, line width=2.0pt]
  table[row sep=crcr]{%
0.322736007494350  0.187529085157431\\
0.181290869729279  0.111533379831984\\
0.102590411249172  0.058781669310102\\
0.056694075094424  0.027029795823344\\
};

\addplot [color=orange, line width=2.0pt]
  table[row sep=crcr]{%
0.322736007494350   0.158409922118760\\
0.181290869729279	0.056721383957223\\
0.102590411249172	0.018791833298399\\
0.056694075094424	0.005855364223746\\
};

\addplot [color=green, line width=2.0pt]
  table[row sep=crcr]{%
0.322736007494350	0.030782209106854\\
0.181290869729279   0.004523929485229\\
0.102590411249172	8.96794311136e-04\\
0.056694075094424   1.32276558454e-04\\
};

\addplot [color=mycolor2, line width=2.0pt]
  table[row sep=crcr]{%
0.322736007494350   0.004273178773731\\
0.181290869729279	4.015369895557630e-04\\
0.102590411249172	4.370693990466424e-05\\
0.056694075094424	3.781005462336538e-06\\
};
   
\addplot [color=blue, line width=2.0pt]
  table[row sep=crcr]{%
0.322736007494350	8.732097009042834e-04\\
0.181290869729279	3.534680062376802e-05\\
0.102590411249172	1.884972755530112e-06\\
0.056694075094424	7.721198294211290e-08\\
};
   
\addplot [color=purple, line width=2.0pt]
  table[row sep=crcr]{%
0.322736007494350	9.557286603164061e-05\\
0.181290869729279	2.297028994540833e-06\\
0.102590411249172	8.757336082139629e-08\\
0.056694075094424	1.783730052077826e-09\\
};

\node[right, align=left, text=black, font=\footnotesize]
at (axis cs:0.1005,0.03) {$1$};

\addplot [color=black, line width=1.5pt]
  table[row sep=crcr]{%
0.100   0.03500\\
0.075   0.02625\\
0.100   0.02625\\
0.100   0.03500\\
};

\node[right, align=left, text=black, font=\footnotesize]
at (axis cs:0.1005,0.008) {$2$};

\addplot [color=black, line width=1.5pt]
  table[row sep=crcr]{%
0.100   0.010\\
0.075   0.005625\\
0.100   0.005625\\
0.100   0.010\\
};

\node[right, align=left, text=black, font=\footnotesize]
at (axis cs:0.1005,2.5e-4) {$3$};

\addplot [color=black, line width=1.5pt]
  table[row sep=crcr]{%
0.100   4.00e-04\\
0.075   1.68e-04\\
0.100   1.68e-04\\
0.100   4.00e-04\\
};

\node[right, align=left, text=black, font=\footnotesize]
at (axis cs:0.1005,9e-6) {$4$};

\addplot [color=black, line width=1.5pt]
  table[row sep=crcr]{%
0.100   2.000e-05\\
0.075   6.328e-06\\
0.100   6.328e-06\\
0.100   2.000e-05\\
};

\node[right, align=left, text=black, font=\footnotesize]
at (axis cs:0.1005,4e-7) {$5$};

\addplot [color=black, line width=1.5pt]
  table[row sep=crcr]{%
0.100   8.000e-07\\
0.075   1.898e-07\\
0.100   1.898e-07\\
0.100   8.000e-07\\
};

\node[right, align=left, text=black, font=\footnotesize]
at (axis cs:0.1005,2e-8) {$6$};

\addplot [color=black, line width=1.5pt]
  table[row sep=crcr]{%
0.100   4.0000e-8\\
0.075   7.7119e-9\\
0.100   7.7119e-9\\
0.100   4.0000e-8\\
};

\end{axis}
\end{tikzpicture}%
}
    \end{subfigure}%
    \begin{subfigure}[b]{0.5\textwidth}
        \resizebox{\textwidth}{!}{\definecolor{mycolor2}{rgb}{0.00000,1.00000,1.00000}%
\pgfplotsset{
  log x ticks with fixed point/.style={
      xticklabel={
        \pgfkeys{/pgf/fpu=true}
        \pgfmathparse{exp(\tick)}%
        \pgfmathprintnumber[fixed  zerofill, precision=2]{\pgfmathresult}
        \pgfkeys{/pgf/fpu=false}
      }
  }
}
\begin{tikzpicture}

\begin{axis}[%
width=3.875in,
height=2.36in,
at={(2.6in,1.099in)},
scale only axis,
xmode=log,
xmin=0.064,
xmax=0.3239,
xminorticks=true,
xlabel = {$h$ [-]},
ylabel = {$||e^{\lambda(T)}-e^{\lambda_h(T)}||_{\Omega}$},
ymode=log,
ymin=1e-12,
ymax=0.003,
yminorticks=true,
axis background/.style={fill=white},
title style={font=\bfseries},
xmajorgrids,
xminorgrids,
ymajorgrids,
yminorgrids,
legend style={legend cell align=left, align=left, draw=white!15!black}
]
                
\addplot [color=red, line width=2.0pt]
  table[row sep=crcr]{%
0.322736007494350  0.002693476063738\\
0.181290869729279  9.59884054064e-04\\
0.102590411249172  3.63982145118e-04\\
0.056694075094424  1.20592832373e-04\\
};
\addlegendentry{$p=1$}

\addplot [color=orange, line width=2.0pt]
  table[row sep=crcr]{%
0.322736007494350   0.001037414246851\\
0.181290869729279	2.36009544431e-04\\
0.102590411249172	4.99995639178e-05\\
0.056694075094424	1.06982486993e-05\\
};
\addlegendentry{$p=2$}
  
\addplot [color=green, line width=2.0pt]
  table[row sep=crcr]{%
0.322736007494350	1.458160611872059e-04\\
0.181290869729279   1.410147730893084e-05\\
0.102590411249172	1.663450250767421e-06\\
0.056694075094424   1.521934292184086e-07\\
};
\addlegendentry{$p=3$}

\addplot [color=mycolor2, line width=2.0pt]
  table[row sep=crcr]{%
0.322736007494350   1.645108891244398e-05\\
0.181290869729279	1.102623133964971e-06\\
0.102590411249172	8.594904745523011e-08\\
0.056694075094424	4.605499850873054e-09\\
};
\addlegendentry{$p=4$}
      
\addplot [color=blue, line width=2.0pt]
  table[row sep=crcr]{%
0.322736007494350	2.472107216747536e-06\\
0.181290869729279	7.268119497403347e-08\\
0.102590411249172	3.413759704328640e-09\\
0.056694075094424	1.002091424789810e-10\\
};
\addlegendentry{$p=5$}

\addplot [color=purple, line width=2.0pt]
  table[row sep=crcr]{%
0.322736007494350	2.337545229693058e-07\\
0.181290869729279	3.316181559449441e-09\\
0.102590411249172	8.529710105553027e-11\\
0.056694075094424	1.204915580463978e-12\\
};
\addlegendentry{$p=6$}

\node[right, align=left, text=black, font=\footnotesize]
at (axis cs:0.1005,0.0002) {$2$};

\addplot [color=black, line width=1.5pt]
  table[row sep=crcr]{%
0.100   0.0002\\
0.075   0.0001125\\
0.100   0.0001125\\
0.100   0.0002\\
};

\node[right, align=left, text=black, font=\footnotesize]
at (axis cs:0.1005,1.5e-5) {$3$};

\addplot [color=black, line width=1.5pt]
  table[row sep=crcr]{%
0.100   2.0000e-5\\
0.075   8.4375e-6\\
0.100   8.4375e-6\\
0.100   2.0000e-5\\
};

\node[right, align=left, text=black, font=\footnotesize]
at (axis cs:0.1005,4e-7) {$4$};

\addplot [color=black, line width=1.5pt]
  table[row sep=crcr]{%
0.100   8.000e-07\\
0.075   2.531e-07\\
0.100   2.531e-07\\
0.100   8.000e-07\\
};

\node[right, align=left, text=black, font=\footnotesize]
at (axis cs:0.1005,2e-8) {$5$};

\addplot [color=black, line width=1.5pt]
  table[row sep=crcr]{%
0.100   5.000e-8\\
0.075   1.186e-8\\
0.100   1.186e-8\\
0.100   5.000e-8\\
};

\node[right, align=left, text=black, font=\footnotesize]
at (axis cs:0.1005,7e-10) {$6$};

\addplot [color=black, line width=1.5pt]
  table[row sep=crcr]{%
0.100   2.000e-09\\
0.075   3.559e-10\\
0.100   3.559e-10\\
0.100   2.000e-09\\
};

\node[right, align=left, text=black, font=\footnotesize]
at (axis cs:0.1005,1e-11) {$7$};

\addplot [color=black, line width=1.5pt]
  table[row sep=crcr]{%
0.100   4.000e-11\\
0.075   5.339e-12\\
0.100   5.339e-12\\
0.100   4.000e-11\\
};

\end{axis}
\end{tikzpicture}%
}
    \end{subfigure}%
    \caption{Test case 1: computed errors and convergence rates in DG-norm (left) and $L^2$-norm (right), $\Delta t = 10^{-6}$.}
    \label{fig:errors2D}
\end{figure}
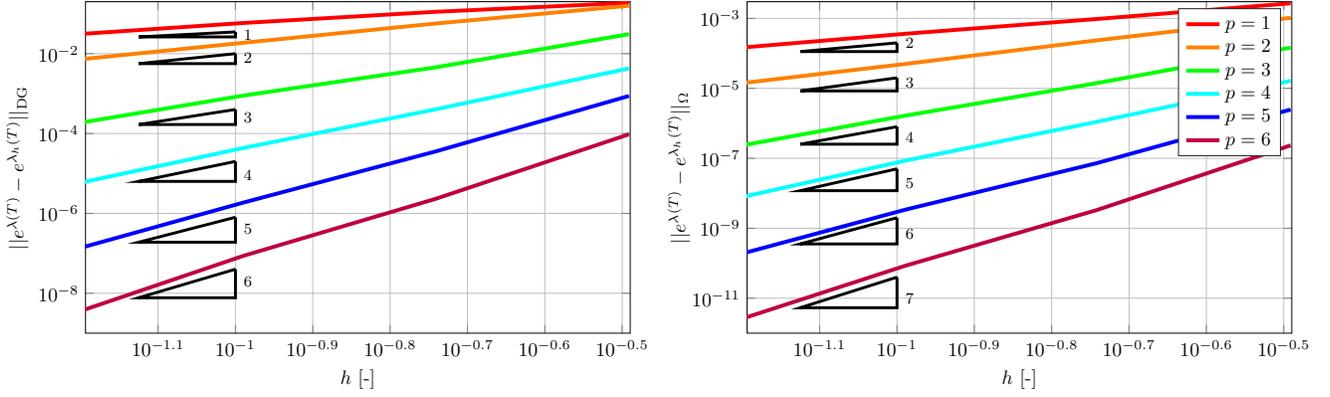
\begin{figure}[t!]
    \begin{subfigure}[b]{0.5\textwidth}
          \resizebox{\textwidth}{!}{\definecolor{EI}{rgb}{0.00000,0.50000,1.00000}%
\definecolor{CN}{rgb}{1.00000,0.00000,0.50000}%
\pgfplotsset{
  log x ticks with fixed point/.style={
      xticklabel={
        \pgfkeys{/pgf/fpu=true}
        \pgfmathparse{exp(\tick)}%
        \pgfmathprintnumber[fixed  zerofill, precision=2]{\pgfmathresult}
        \pgfkeys{/pgf/fpu=false}
      }
  }
}
\begin{tikzpicture}

\begin{axis}[%
width=3.875in,
height=2.36in,
at={(2.6in,1.099in)},
scale only axis,
xmode=log,
xmin=0.0125,
xmax=0.1000,
xminorticks=true,
xlabel = {$\Delta t$ [-]},
ylabel = {$||e^{\lambda(T)}-e^{\lambda_h(T)}||$},
ymode=log,
ymin=1e-7,
ymax=0.01,
yminorticks=true,
axis background/.style={fill=white},
title style={font=\bfseries},
title={Convergence with respect to the timestep $\Delta t$},
xmajorgrids,
xminorgrids,
ymajorgrids,
yminorgrids,
legend style={legend cell align=left, align=left, draw=white!15!black}
]
                
\addplot [color=CN, dashed, line width=2.0pt]
  table[row sep=crcr]{%
0.1000  9.15859202703e-05\\
0.0500  2.29548010472e-05\\
0.0250  5.74335419942e-06\\
0.0125  1.44530499597e-06\\
};
\addlegendentry{$\vartheta=1/2$: $\mathrm{DG}$-norm}

\addplot [color=CN, line width=2.0pt]
  table[row sep=crcr]{%
0.1000  2.02052524445e-05\\
0.0500  5.05236927068e-06\\
0.0250  1.26309239233e-06\\
0.0125  3.15779417345e-07\\
};
\addlegendentry{$\vartheta=1/2$: $L^2$-norm}

\addplot [color=EI, dashed, line width=2.0pt]
  table[row sep=crcr]{%
0.1000  0.007541148224797\\
0.0500  0.003702862354079\\
0.0250  0.001834823252392\\
0.0125  0.000913302282371\\
};
\addlegendentry{$\vartheta=1$: $\mathrm{DG}$-norm}
 
\addplot [color=EI, line width=2.0pt]
  table[row sep=crcr]{%
0.1000  0.001654753912983\\
0.0500  0.000812476266045\\
0.0250  0.000402583319779\\
0.0125  0.000200387356979\\
};
\addlegendentry{$\vartheta=1$: $L^2$-norm}
 
\node[right, align=left, text=black, font=\footnotesize]
at (axis cs:0.0225,0.00018) {$1$};

\addplot [color=black, line width=1.5pt]
  table[row sep=crcr]{%
0.022   0.00022\\
0.015   0.00015\\
0.022   0.00015\\
0.022   0.00022\\
};

\node[right, align=left, text=black, font=\footnotesize]
at (axis cs:0.0225,3.5e-7) {$2$};

\addplot [color=black, line width=1.5pt]
  table[row sep=crcr]{%
0.022   5.0e-7\\
0.015   2.3e-7\\
0.022   2.3e-7\\
0.022   5.0e-7\\
};

\end{axis}
\end{tikzpicture}%
}
         \caption{Convergence in time with $p=1$.}
          \label{fig:errors2Dtime}
    \end{subfigure}%
    \begin{subfigure}[b]{0.5\textwidth}
        \resizebox{\textwidth}{!}{\definecolor{EI}{rgb}{0.00000,0.50000,1.00000}%
\definecolor{CN}{rgb}{1.00000,0.00000,0.50000}%

\begin{tikzpicture}
\begin{axis}[%
width=3.875in,
height=2.36in,
at={(2.6in,1.099in)},
scale only axis,
xmin=1,
xmax=8,
xlabel style={font=\color{white!15!black}},
xlabel={$p$},
ymode=log,
ymin=1e-10,
ymax=2.5,
yminorticks=true,
ylabel style={font=\color{white!15!black}},
ylabel={$||e^{\lambda(T)}-e^{\lambda_h(T)}||$},
axis background/.style={fill=white},
title={Convergence with respect to the degree $p$},
xmajorgrids,
ymajorgrids,
yminorgrids,
legend style={legend cell align=left, align=left, draw=white!15!black}
]

\addplot [color=CN, dashed, line width=2.0pt]
  table[row sep=crcr]{%
1	0.187529085157431\\
2	0.158409922118760\\
3	0.030782209106854\\
4	0.004273178773731\\
5   0.000873209700904\\
6	0.000095572866032\\
7   0.000022704839818\\
8   0.000002008866318\\
};
\addlegendentry{$\vartheta=1/2$: $\mathrm{DG}$-norm}

\addplot [color=CN, line width=2.0pt]
  table[row sep=crcr]{%
1	0.002693476063738\\
2	0.001037414246851\\
3	0.000145816061187\\
4   0.000016451088912\\
5   0.000002472107217\\
6	0.000000233754523\\
7   0.000000040090218\\
8   0.000000003139794\\
};
\addlegendentry{$\vartheta=1/2$: $L^2$-norm}
   

\end{axis}
\end{tikzpicture}%
}
         \caption{Convergence in $p$ with $\Delta t = 10^{-6}$.}
        \label{fig:errors2Dpoly}
    \end{subfigure}%
    \caption{Test case 1: computed errors and convergence rates in convergence concerning timestep (left) and polynomial degree (right) with $N_\mathrm{el} = 30$.}
\end{figure}
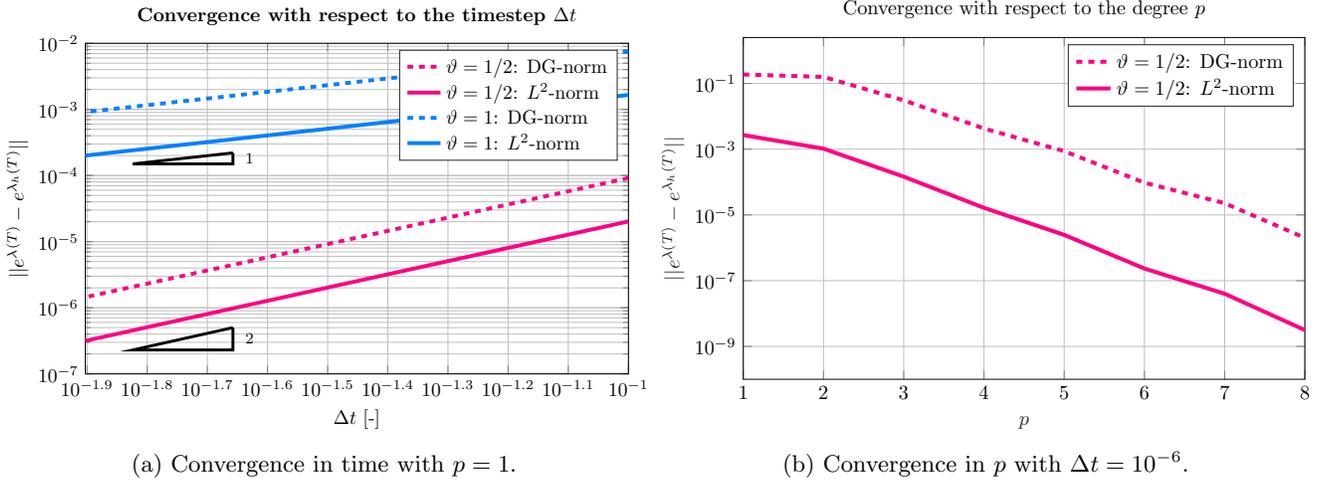
For the numerical tests in this section, we use Lymph library \cite{lymphpaper} to solve the FK equation $(d=2)$. We define the domain $\Omega=(0,1)^2$, which we discretize by means of a polygonal mesh obtained by using PolyMesher \cite{Polymesher}. Concerning the time discretization, we use a timestep $\Delta t = 10^{-6}$ and the final time $T=2\times 10^{-5}$. We consider the following manufactured exact solution:
\begin{equation}
    \lambda(x,y,t) = \log\left((\cos(\pi x)\cos(\pi y)+2) e^{-t}\right).
\end{equation}
We fix the physical parameters as follows: $\mathbf{D}=\mathbf{I}$ and $\alpha=0.1$. The forcing term and the Dirichlet boundary conditions are derived accordingly. To solve the resulting nonlinear system we adopt a Newton method with tolerance equal to $\epsilon = 10^{-10}$.
\par
In Figure~\ref{fig:errors2D}, we report the computed errors in both the DG and $L^2$ norms at the final time. We have performed the convergence test keeping fixed the polynomial order of the space approximation $p_K=p=1,...,6$ $\forall K \in \partition$ and using different mesh refinements $(N_\mathrm{el}= 30,100,300,1000)$. It can be observed that the slope of error decrease is equal to the polynomial degree $p$ for the DG-norm and equal to $p+1$ for the $L^2$-norm.
\par
In Figure~\ref{fig:errors2Dtime}, we report the computed errors with respect to the timestep $\Delta t$, using both Crank-Nicolson $(\vartheta=0.5)$ and the Implicit Euler $(\vartheta=1)$ schemes. The space discretization is computed on a mesh of $N_\mathrm{el}=1000$ elements and with polynomial degree $p=6$. As expected, the use of the Crank-Nicolson method leads to a second-order convergence whereas the error decays with a first-order rate if the implicit Euler scheme is employed. We remark that the case $\vartheta=1$ is fully covered by our theoretical analysis, whereas the proof of convergence for $\vartheta=1/2$ is under investigation.
\par
A convergence analysis with respect to the polynomial order $p$ is also performed on a coarse mesh of 30 elements and with a time integration based on the Crank-Nicolson scheme with timestep $\Delta t = 10^{-6}$. The results are reported in Figure~\ref{fig:errors2Dpoly}, where can observe exponential convergence can be observed.

\subsection{Test case 2: Travelling waves in two dimensions}
In this section, we exployt the positivity-preserving PolyDG formulation to simulate a traveling-wave solution of the FK equation in 2D, with the aim of comparing the formulation we propose in this article, with the (non positivity-preserving) scheme introduces in \cite{corti:FK}. The manufactured solution is of the form:
\begin{equation}
    e^{\lambda(x,y,t)} = c(x,y,t) = \psi(x-vt) = \psi(\xi).
\end{equation}
By substituting it in Equation \eqref{eq:fk_strong} with $f=0$, we obtain the following equivalent system of ordinary differential equations:
\begin{equation}
    \begin{dcases}
        \chi'(\xi) = -\dfrac{v}{d}\chi(\xi) + \dfrac{1}{d}\psi(\xi)(\psi(\xi)-1) & \xi\in(0,T), \\
        \psi'(\xi) = \chi(\xi) & \xi\in(0,T), \\
    \end{dcases}
\end{equation}
where we have used the assumption of isotropic diffusion tensor $\mathbf{D}=d\mathbf{I}$. In particular, we fix $d=10^{-3}$, $\alpha=1$ and penalty parameter $\eta_0=1$. Concerning the wave's parameters we take the speed $v=0.1$ and the initial data $\psi(0)=1$ and $\chi(0)=-10^{-2}$. We consider a rectangle $\Omega=(0, 5)\times(0, 1)$ as domain. We present the results of two simulations, with different final times $T=5$ and $T=10$ and timestep $\Delta t = 10^{-2}$. Concerning the nonlinear Newton solver, we fix a tolerance $\epsilon=10^{-6}$. 
\par
\begin{table}[t]
    \centering
    \begin{tabular}{c|c|C|C|c|C|C}
    \hline
    \multicolumn{7}{c}{\textbf{Positivity-preserving method}} \\
    \hline
    \hline
    \textbf{Method}
    & \multicolumn{3}{c|}{\textbf{$h=0.72802$}} 
    & \multicolumn{3}{c}{\textbf{$h=0.41057$}}
    \\    \hline
    $p$  
    & \textbf{DOFs} & $T = 5$ & $T = 10$ 
    & \textbf{DOFs} & $T = 5$ & $T = 10$ 
    \\ \hline
    $\boldsymbol{p=1}$ 
    & $90$   & $1.63\times10^{0}$  & $5.71\times10^{-1}$  
    & $300$  & $1.05\times10^{-2}$ & $3.34\times10^{-2}$ 
    \\ \hline
    $\boldsymbol{p=2}$
    & $180$  & $7.02\times10^{-2}$ & $4.67\times10^{-2}$ 
    & $600$  & $2.89\times10^{-3}$ & $5.35\times10^{-3}$ 
    \\ \hline
    $\boldsymbol{p=3}$
    & $300$  & $7.54\times10^{-2}$ & $5.77\times10^{-2}$ 
    & $1000$ & $4.52\times10^{-2}$ & $1.03\times10^{-1}$ 
    \\ \hline    
    $\boldsymbol{p=4}$
    & $450$  & $1.20\times10^{-2}$ & $5.67\times10^{-2}$ 
    & $1500$ & $1.37\times10^{-1}$  & $1.76\times10^{-1}$ 
    \\ \hline    
    \end{tabular}

    \begin{tabular}{c|c|C|C|c|C|C}
    \hline
    \multicolumn{7}{c}{\textbf{DG method \cite{corti:FK}}}  \\
    \hline
    \hline
    \textbf{Method}
    & \multicolumn{3}{c|}{\textbf{$h=0.72802$}} 
    & \multicolumn{3}{c}{\textbf{$h=0.41057$}}
    \\    \hline
    $p$  
    & \textbf{DOFs} & $T = 5$ & $T = 10$ 
    & \textbf{DOFs} & $T = 5$ & $T = 10$   
    \\ \hline
    $\boldsymbol{p=1}$ 
    & $90$   & $6.38\times10^{3}$ & $9.44\times10^{3}$
    & $300$  & $2.19\times10^{4}$ & $2.30\times10^{4}$
    \\ \hline
    $\boldsymbol{p=2}$
    & $180$  & $2.48\times10^{0}$  & $1.06\times10^{5}$ 
    & $600$  & $1.76\times10^{0}$  & $1.01\times10^{4}$ 
    \\ \hline
    $\boldsymbol{p=3}$
    & $300$  & $1.07\times10^{-1}$ & $1.29\times10^{5}$
    & $1000$ & $7.80\times10^{-3}$ & $1.79\times10^{-1}$ 
    \\ \hline    
    $\boldsymbol{p=4}$
    & $450$  & $1.54\times10^{-2}$  & $6.47\times10^{-1}$  
    & $1500$ & $8.14\times10^{-4}$  & $4.50\times10^{-3}$
    \\ \hline    
    \end{tabular}
    \caption{Comparison of the computed errors in the DG-norm based on employing the proposed positivity-preserving scheme and the DG method of \cite{corti:FK}, for different polynomial degrees, different mesh sizes, and different final times.}
    \label{tab:errorstimerefL2}
\end{table}

\begin{table}[t]
    \centering
    \begin{tabular}{c|c|C|C|c|C|C}
    \hline
    \multicolumn{7}{c}{\textbf{Positivity-preserving method}} \\
    \hline
    \hline
    \textbf{Method}
    & \multicolumn{3}{c|}{\textbf{$h=0.72802$}} 
    & \multicolumn{3}{c}{\textbf{$h=0.41057$}}
    \\    \hline
    $p$ 
    & \textbf{DOFs} & $T = 5$ & $T = 10$ 
    & \textbf{DOFs} & $T = 5$ & $T = 10$ 
    \\ \hline
    $\boldsymbol{p=1}$ 
    & $90$   & $7.53\times10^{0}$  & $1.71\times10^{-1}$  
    & $300$  & $2.09\times10^{-1}$ & $1.05\times10^{-1}$ 
    \\ \hline
    $\boldsymbol{p=2}$
    & $180$  & $2.63\times10^{-2}$ & $1.83\times10^{-2}$ 
    & $600$  & $1.83\times10^{-3}$ & $2.66\times10^{-3}$ 
    \\ \hline
    $\boldsymbol{p=3}$
    & $300$  & $3.35\times10^{-2}$ & $1.99\times10^{-2}$ 
    & $1000$ & $9.24\times10^{-2}$ & $2.03\times10^{-1}$ 
    \\ \hline    
    $\boldsymbol{p=4}$
    & $450$  & $9.37\times10^{-3}$ & $4.24\times10^{-2}$ 
    & $1500$ & $2.63\times10^{-1}$ & $3.21\times10^{-1}$
    \\ \hline    
    \end{tabular}

    \begin{tabular}{c|c|C|C|c|C|C}
    \hline
    \multicolumn{7}{c}{\textbf{DG method \cite{corti:FK}}}  \\
    \hline
    \hline
    \textbf{Method}
    & \multicolumn{3}{c|}{\textbf{$h=0.72802$}} 
    & \multicolumn{3}{c}{\textbf{$h=0.41057$}}
    \\    \hline
    $p$  
    & \textbf{DOFs} & $T = 5$ & $T = 10$ 
    & \textbf{DOFs} & $T = 5$ & $T = 10$    
    \\ \hline
    $\boldsymbol{p=1}$ 
    & $90$   & $9.02\times10^{4}$ & $1.45\times10^{5}$ 
    & $300$  & $7.89\times10^{5}$ & $8.22\times10^{5}$ 
    \\ \hline
    $\boldsymbol{p=2}$
    & $180$  & $1.13\times10^{0}$  & $2.99\times10^{5}$
    & $600$  & $8.05\times10^{-1}$ & $2.56\times10^{5}$ 
    \\ \hline
    $\boldsymbol{p=3}$
    & $300$  & $4.76\times10^{-1}$ & $1.33\times10^{6}$
    & $1000$ & $1.05\times10^{-1}$ & $1.05\times10^{-1}$ 
    \\ \hline    
    $\boldsymbol{p=4}$
    & $450$  & $2.00\times10^{-1}$  & $3.17\times10^{-1}$   
    & $1500$ & $1.67\times10^{-2}$  & $1.47\times10^{-2}$
    \\ \hline    
    \end{tabular}
    \caption{Comparison of the computed errors in the DG-norm based on employing the proposed positivity-preserving scheme and the DG method of \cite{corti:FK}, for different polynomial degrees, different mesh sizes, and different final times.}
    \label{tab:errorstimerefDG}
\end{table}
\begin{figure}[t]
	\centering
	{\includegraphics[width=\textwidth]{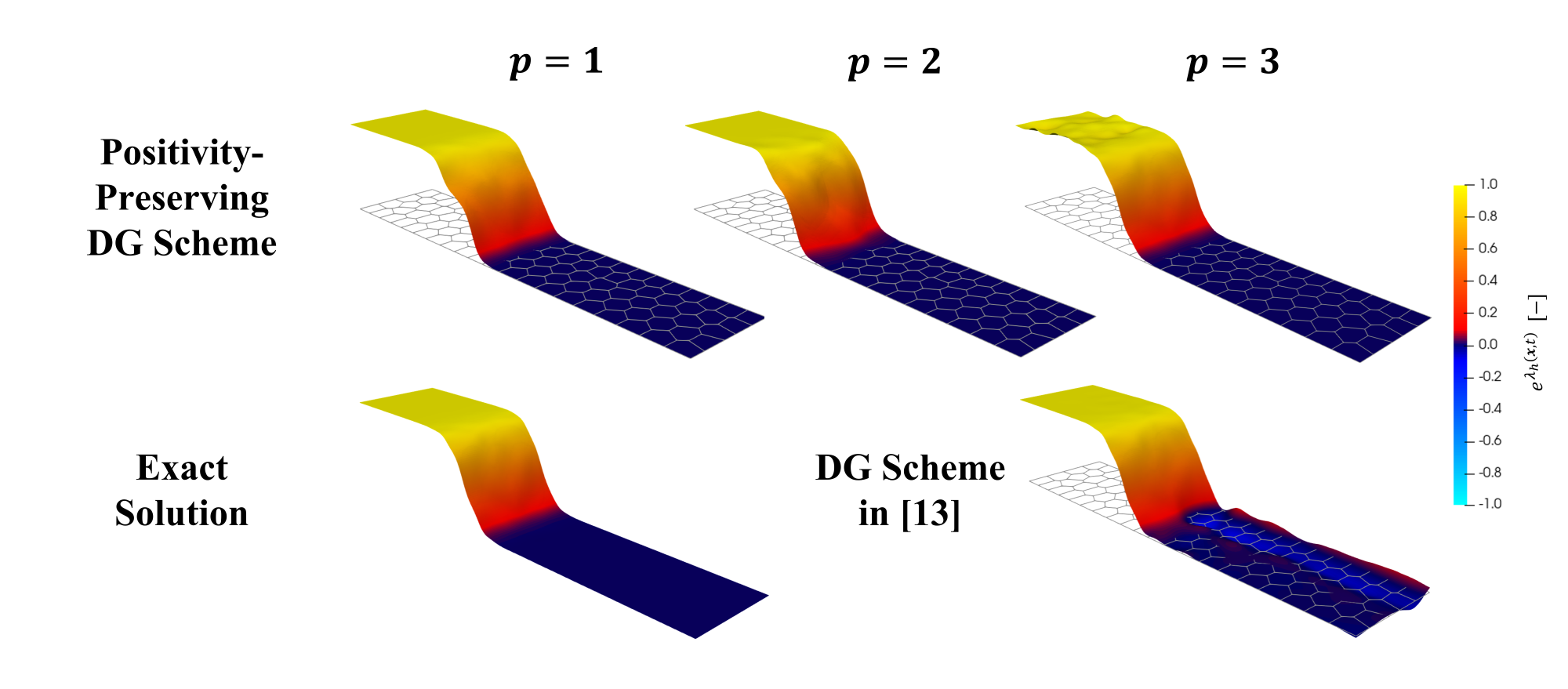}}
	\caption{Test case 2: Comparison of the numerical solutions computed both using the proposed positivity-preserving DG scheme and the DG method presented in \cite{corti:FK}, with the exact solution at time $T=10$.}
	\label{fig:waves2D}
\end{figure}
In Tables~\ref{tab:errorstimerefL2} and~\ref{tab:errorstimerefDG}, we report the computed errors in the $L^2$ and DG norms computed at the final times $T=5$ and $T=10$, respectively. In particular, we compare the results obtained by using our positivity-preserving method \eqref{eq:lambda_fk_fully} and the DG method proposed in \cite{corti:FK}, with a semi-implicit treatment of the nonlinearity and a penalty parameter $\eta=10$. We can observe that, also using low order polynomials $(p=1)$, our method is able to correctly represent the wave propagation front and leads to smaller errors (one order of magnitude). On the contrary, the method in \cite{corti:FK} fails to correctly simulate the wavefront because it does not preserve the positivity of the solution and the equilibrium $c=0$ is unstable. 
\par
Moreover, from the results of Table~\ref{tab:errorstimerefL2} and Table~\ref{tab:errorstimerefDG}, we can observe for $p=4$ and $T=10$ that the proposed positivity-preserving scheme does not lead to a reduction of the error compared with the results obtained for $p=3$. Indeed, we can observe in Figure~\ref{fig:waves2D} that for $p=3$ we have the formation of some small oscillations around the equilibrium $c=1$. This is probably due to Newton's iterations that might be badly conditioned for large values of polynomial degrees. The effect of this problem cannot be observed in the method of \cite{corti:FK}, but in this case, the positivity of the solution cannot be guaranteed.

\section{Numerical results: brain applications}
\label{sec:numericalresultsbrain}
In this section, we present the numerical results obtained in two different test cases: a two-dimensional simulation of a sagittal section of a brain and a three-dimensional simulation of brain geometries reconstructed from Magnetic Resonance Images (MRI).
\par
In the prions' spreading applications, the diffusion tensor is typically modelled as the superimposition of an extracellular diffusion effect with magnitude $d_\mathrm{ext}$ and an axonal diffusion with magnitude $d_\mathrm{axn}$ \cite{weickenmeierPhysicsbasedModelExplains2019}; for this reason, in this section, we assume that $\mathbf{D}$ has the following structure:
\begin{equation}
\label{eq:difftensor}
    \mathbf{D} = d_\mathrm{ext}\mathbf{I} + d_\mathrm{axn}(\boldsymbol{n}\otimes \boldsymbol{n}),
\end{equation}
where $\boldsymbol{n}=\boldsymbol{n}(\boldsymbol{x})$ is the axonal fibres direction in the point $\boldsymbol{x}\in\Omega$ and $d_\mathrm{ext}, d_\mathrm{axn} \geq 0$. The axonal direction is derived from Diffusion Weighted Imaging (DWI) and represents the principal orientation of the connections between the neurons (axons). Most of the spreading of the prions seems to happen through the axons \cite{weickenmeierPhysicsbasedModelExplains2019}, however, due to the brain structure, this is true only in white matter, while in grey matter, the diffusion can be considered to be isotropic. 
\par
In order to construct the axonal component of the diffusion tensor $\mathbf{D}$, we derive the diffusion tensor from DWI medical images by using Freesurfer and Nibabel \cite{Nibabel}. The principal eigenvector $\boldsymbol{n}$ of the tensor is then computed elementwise to find the diffusion tensor in Equation \eqref{eq:difftensor}. We refer to \cite{Mardal:Mesh} for more details on the reconstruction of $\mathbf{D}$ starting from medical images. Concerning the forcing term we fix $f=0$ and we impose homogeneous Neumann boundary conditions in both test cases.
\par
Concerning test case 3 in Section~\ref{sec:parkinson_test}, we simulate the spreading of $\alpha$-Synuclein in Parkinson's disease in a two-dimensional brain section. The simulation starts with a concentration of the misfolded proteins only at the base of the brainstem, so an initial stage of the pathology, and it requires many years ($\simeq 25$ years) of development. On the contrary, test case 4 in Section~\ref{sec:alzheimer_test}, refers to Alzheimers's disease in a three-dimensional brain. The initial concentration is diffused and derived from a Positron Emission Tomography (PET) image of an 83 years old patient with advanced pathological symptoms. 
\subsection{Test case 3: spreading of $\alpha$-Synuclein in a two-dimensional brain section}
\label{sec:parkinson_test}
\begin{figure}[t]
     \centering
     \begin{subfigure}[b]{0.32\textwidth}
         \centering
         \includegraphics[width=\textwidth]{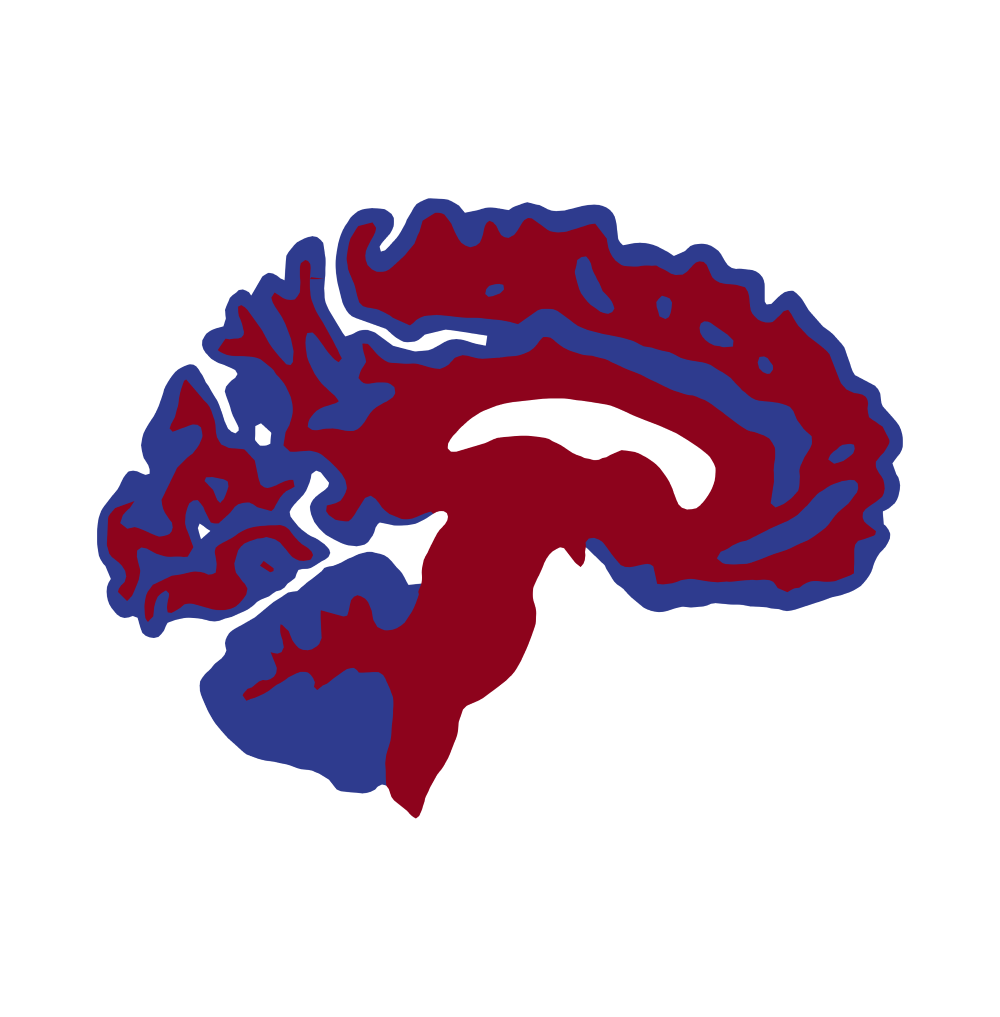}
         \caption{White/Gray matter subdivision}
         \label{fig:BrainSection:GMWM}
     \end{subfigure}
     \hfill
     \begin{subfigure}[b]{0.32\textwidth}
         \centering
         \includegraphics[width=\textwidth]{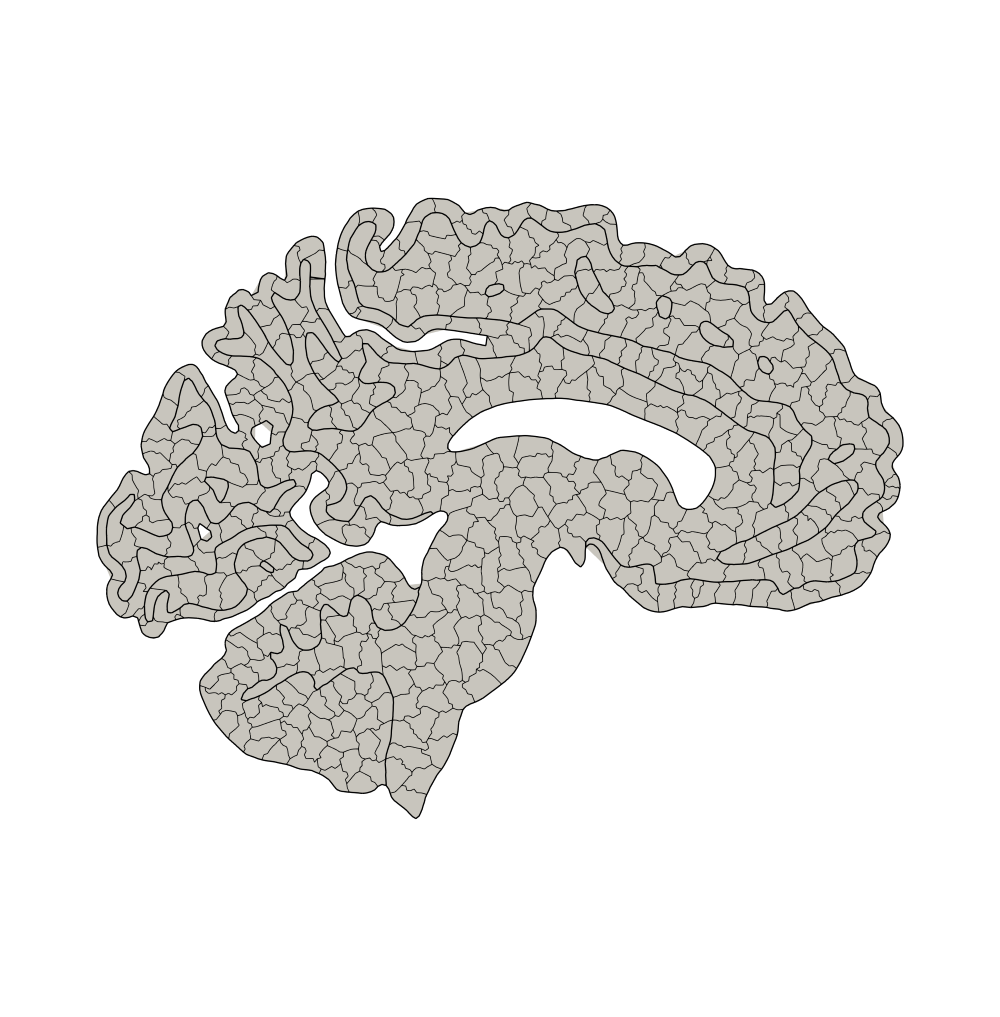}
         \caption{Mesh}
         \label{fig:BrainSection:Mesh}
    \end{subfigure}
    \hfill
    \begin{subfigure}[b]{0.32\textwidth}
         \centering
         \includegraphics[width=\textwidth]{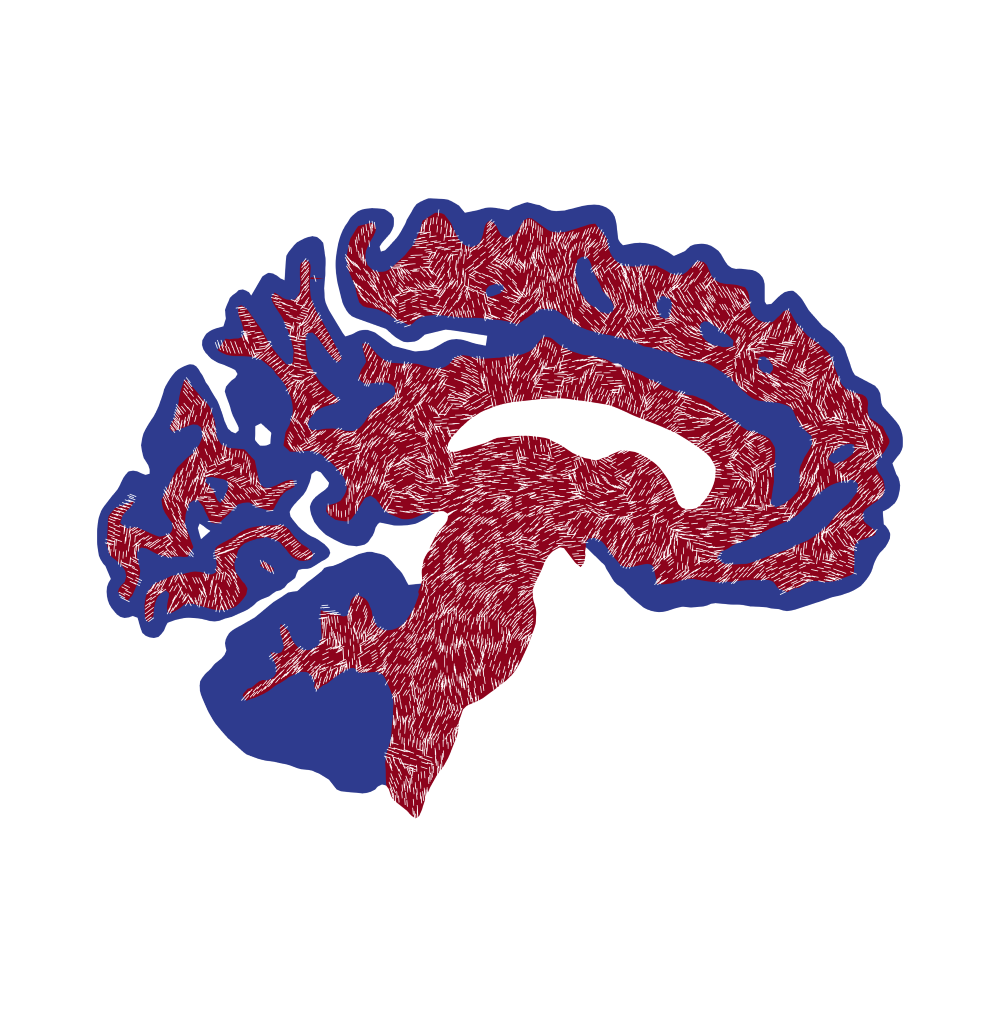}
         \caption{Axonal directions}
         \label{fig:BrainSection:Axn}
    \end{subfigure}
        \caption{Brain section with the specification of the white-gray matter subdivision (a), the agglomerated mesh (b), and the axonal directions (c).}
        \label{fig:BrainSection}
\end{figure}
In this section, we address a numerical simulation of the spreading of $\alpha$-Synuclein in Parkinson's disease on a polygonal agglomerated grid of a sagittal 2D brain section. The geometry is segmented from a structural MRI of a brain from the OASIS-3 database \cite{OASIS3} by means of Freesurfer \cite{Freesurfer}.
The construction of the final mesh of a slice of the brain is performed by using VMTK \cite{VMTK:antiga}. The resulting triangular mesh is composed of $43\,402$ triangles, and each element of the mesh is labelled to be in white or grey matter, according to the MRI segmentation, as in Figure~\ref{fig:BrainSection:GMWM}. However, the generality of the PolyDG method allows us to use mesh elements of any shape and the use of a smaller number of elements allows saving computational cost. For this reason, by using ParMETIS \cite{Parmetis}, we agglomerate the initial triangular mesh into a polygonal mesh of 534 elements, as shown in Figure~\ref{fig:BrainSection:Mesh}. In particular, the agglomeration procedure is performed in a segregated way for the white and the grey matter, in this way we are sure to correctly describe both the domain boundary and the interface between grey/white matters. Finally, in Figure~\ref{fig:BrainSection:Axn}, we report the axonal directions computed in the white matter starting from DWI.
\par
Concerning the physical parameters, we fix the reaction coefficient $\alpha = 0.45/\mathrm{year}$ in grey matter, and  $\alpha = 0.9/\mathrm{year}$ in white matter \cite{schaferInterplayBiochemicalBiomechanical2019}. Moreover, we impose a constant isotropic diffusion $d_\mathrm{ext} = 8\,\mathrm{mm}^2/\mathrm{year}$, and axonal diffusion which is 10 times faster than the isotropic one in the white matter ($d_\mathrm{axn} = 80\,\mathrm{mm}^2/\mathrm{year}$) and is negligible in the grey matter ($d_\mathrm{axn} = 0\,\mathrm{mm}^2/\mathrm{year}$) \cite{schaferInterplayBiochemicalBiomechanical2019}. In this simulation, we fix $\Delta t=0.01\,\mathrm{years}$ and $p=1$, moreover the penalty parameter $\eta_0=1$.
\begin{figure}[t!]
    \centering
    {\includegraphics[width=\textwidth]{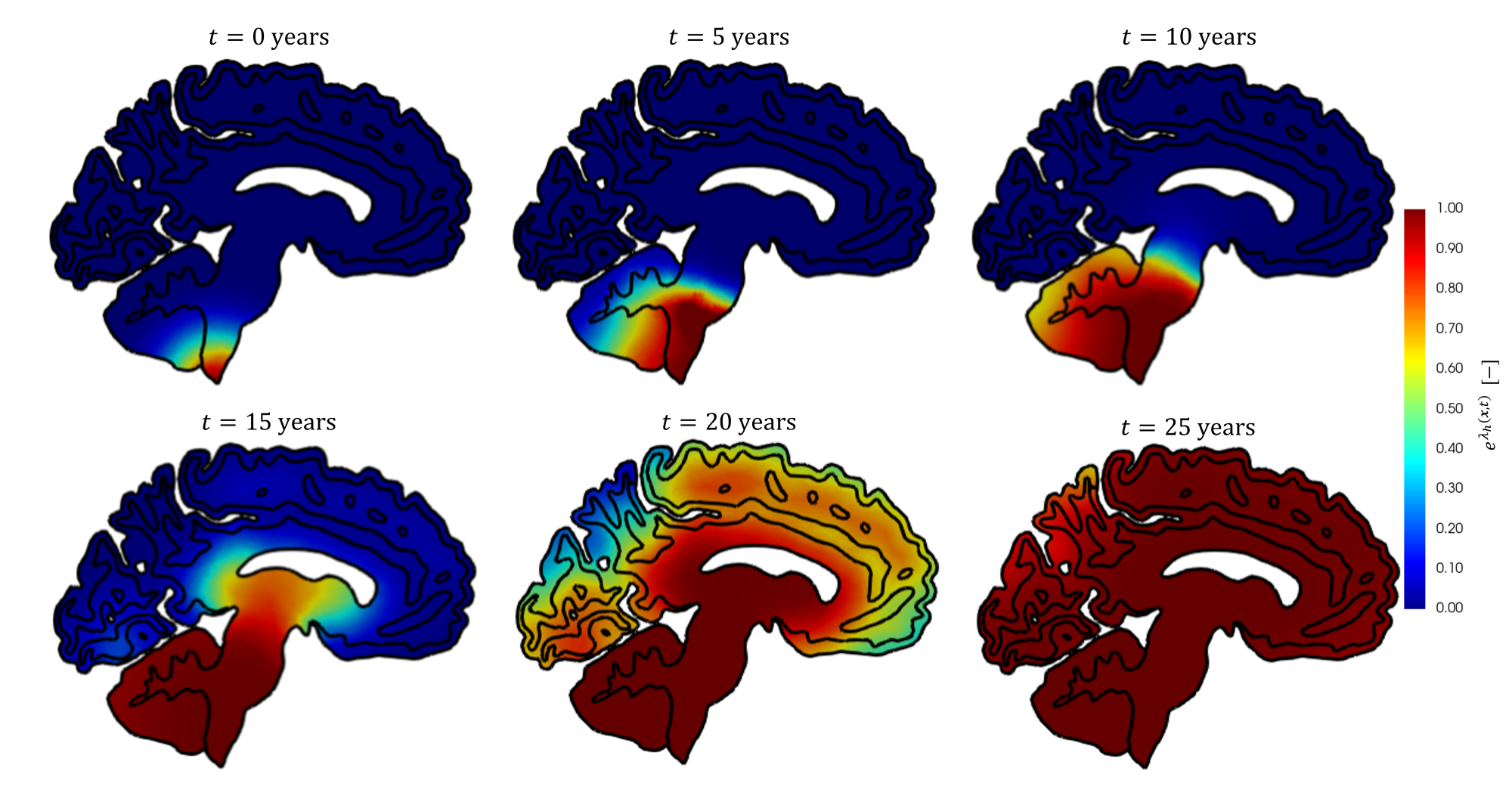}}
    \caption{Patterns of $\alpha$-synuclein concentration at different stages of the pathology.} 
    \label{fig:Solution2DPark}
\end{figure}
\par
The simulation of $\alpha$-Synuclein diffusion in Parkinson’s disease starts from an initial condition, with concentration located in the dorsal motor nucleus \cite{braakStagingBrainPathology2003}. In Figure~\ref{fig:Solution2DPark}, we report both the initial condition and the computed solution at different times $t=0,5,10,15,20,25$ years. First of all, it can be observed that the directions of protein propagations are coherent with the medical literature \cite{Goedert2015}. Indeed, the activation of brain regions follows the Braak staging theory \cite{braakStagingBrainPathology2003}. Moreover, we can notice that the heterogeneity of the reaction parameters causes an earlier activation of the white matter in general, which is clearly visible in the frontal cortex at time $t=20$. By making a comparison with the literature results of \cite{corti:FK}, we have that the reduced reactivity and diffusion inside grey matter causes a slowing of the disease progression times, starting with the same initial condition and an agglomerated mesh with comparable refinement level.

\begin{figure}[t!]
    \begin{subfigure}[b]{0.65\textwidth}
          \resizebox{\textwidth}{!}{\input{Images/Concentrations2D.tikz}}         
          \caption{Mean value of the concentration inside the brain, white, and grey matter}
         \label{fig:Indicator2D:AvgConcentrations}
    \end{subfigure}%
    \begin{subfigure}[b]{0.35\textwidth}
        \includegraphics[width=\textwidth]{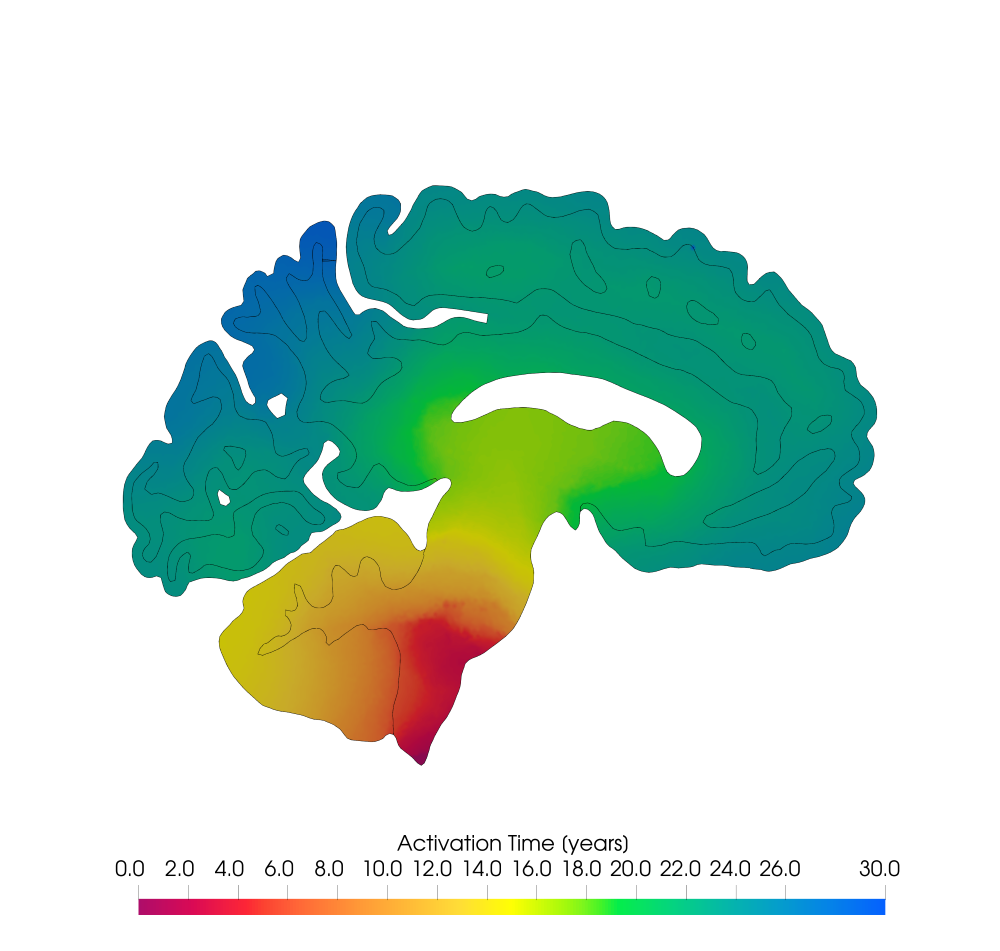}         
        \caption{Activation time on the brain section}
         \label{fig:Indicator2D:ActTimeWhite}
    \end{subfigure}%
    \caption{Test case 3: mean value of the concentration (a) and activation time (b).}
    \label{fig:Indicator2D}
\end{figure}
\par
In Figure~\ref{fig:Indicator2D:AvgConcentrations}, we report the average concentration of misfolded protein $\overline{e^{\lambda_h(t)}}$ inside the brain with respect to the time $t$. Moreover, we compute the average concentrations in white and grey matter separately. As we can observe, in the first years, the increase in the concentration is almost equivalent for the two regions, after 14 years we have a clear distinction. In particular, the higher reactivity and diffusion of the white matter tissue causes a faster increase in the concentration. Moreover, we compute the activation time of the pathology as:
\begin{equation}
\label{eq:acttime}
    \hat{t}(\boldsymbol{x},t) = \chi_{\{e^{\lambda_h(\boldsymbol{x},t)}>c_\mathrm{crit}\}} (\boldsymbol{x},t) \qquad \boldsymbol{x}\in \Omega \quad t\in[0,T],
\end{equation}
where $\chi$ is the indicator function and $c_\mathrm{crit}=0.95$ is the critical value of $\alpha$-Synuclein concentration. We report the computed activation time in Figure~\ref{fig:Indicator2D:ActTimeWhite}. From a pathological perspective, high concentrations of $\alpha$-Synuclein alter the electric signal transport. The indicator \eqref{eq:acttime} measures the time after which a region of the brain can be affected by pathological electric stimuli. The result is qualitatively similar to the literature results \cite{weickenmeierPhysicsbasedModelExplains2019, corti:FK}. Comparing the result with respect to \cite{corti:FK}, we can notice a longer activation time, due to the reduced reactivity and diffusion in grey matter, introduced in this work.

\subsection{Test case 4: spreading of Amyloid-$\beta$ in a three-dimensional brain}
\label{sec:alzheimer_test}
In this section, we present a numerical simulation of the spreading of the Amyloid-$\beta$ on a three-dimensional domain, reconstructed starting from an MRI taken from OASIS-3 database \cite{OASIS3}. The medical images are associated with an 83 years old patient, which is diagnosed to be affected by Alzheimer's disease at the moment of the acquisition. The geometry is segmented by means of Freesurfer \cite{Freesurfer} and then is used to construct a mesh grid of 323'014 tetrahedral elements, using SVMTK library \cite{Mardal:Mesh}. The resulting mesh is reported in Figure~\ref{fig:3D:Mesh}. The problem is solved with the use of a FEniCS code \cite{FenicsCode} (version 2019).
\par
\begin{figure}[t!]
    \begin{subfigure}[b]{0.4\textwidth}
	   {\includegraphics[width=\textwidth]{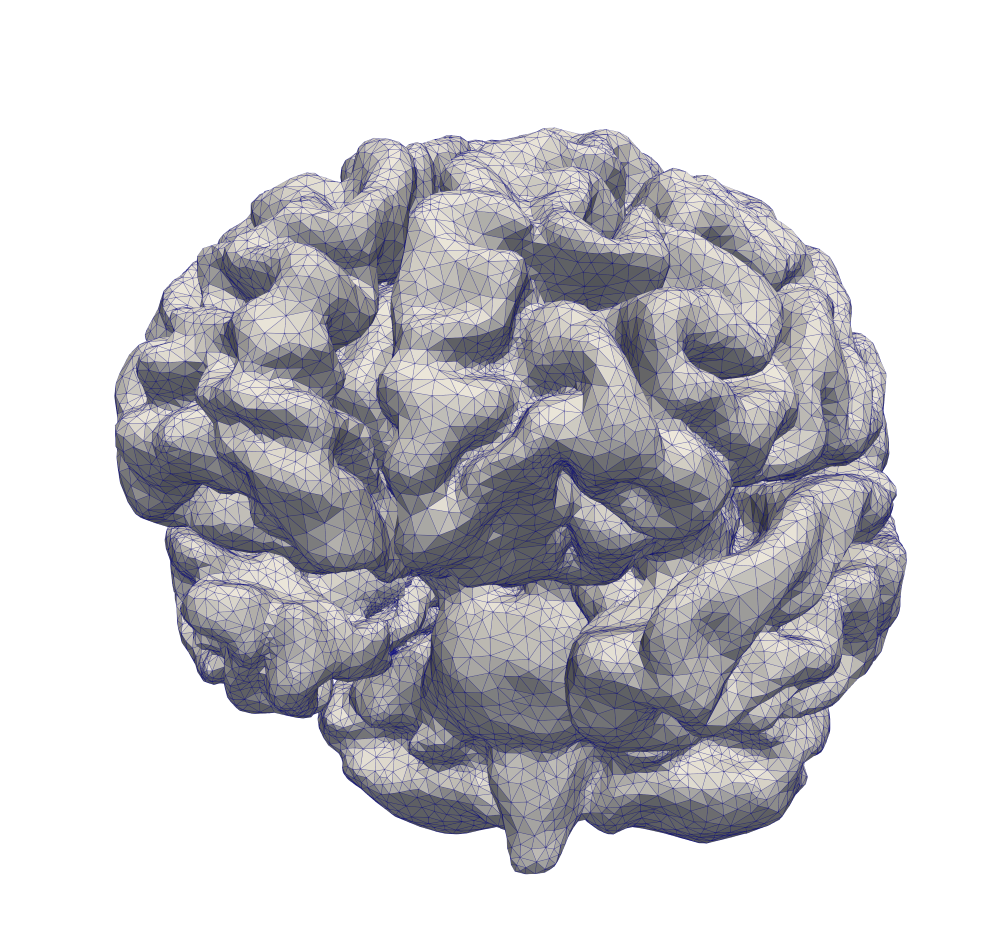}}       
          \caption{Three-dimensional tetrahedral brain mesh}
         \label{fig:3D:Mesh}
    \end{subfigure}%
    \begin{subfigure}[b]{0.6\textwidth}
        \includegraphics[width=\textwidth]{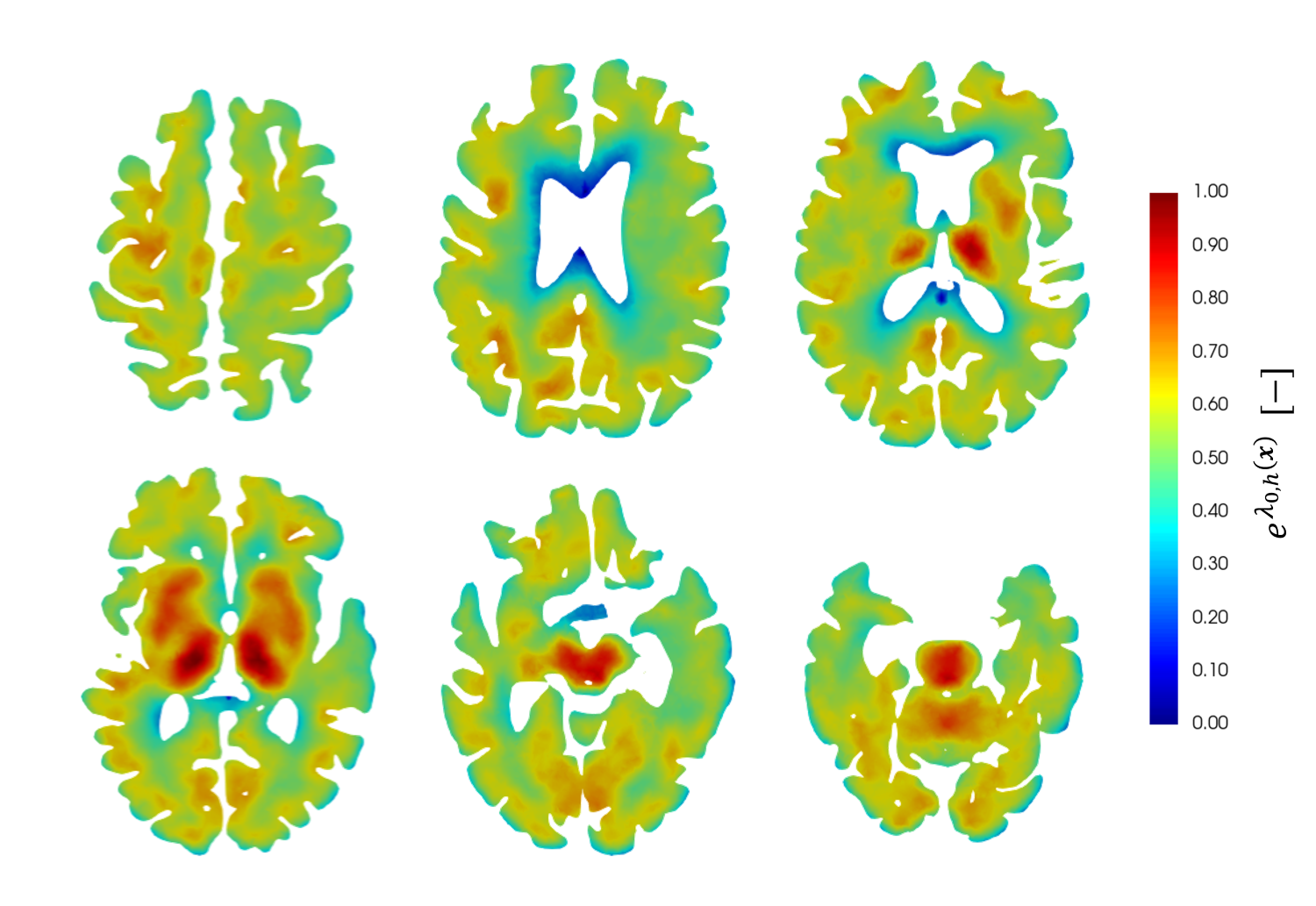}         
        \caption{Initial condition from PET images on brain slices in horizontal plane}
         \label{fig:3D:PET}
    \end{subfigure}%
    \caption{Test case 4: three-dimensional brain mesh (a) and projected initial condition from PET image (b).}
    \label{fig:3D}
\end{figure}
\par
Concerning the parameters of the model, in this test case, for simplicity, we do not make any distinction between white and grey matters, choosing $\alpha = 0.9/\mathrm{year}$, $d_\mathrm{ext} = 8\,\mathrm{mm}^3/\mathrm{year}$, and $d_\mathrm{axn} = 80\,\mathrm{mm}^3/\mathrm{year}$ \cite{schaferInterplayBiochemicalBiomechanical2019, corti:FK}. 
\par
To set up the initial condition for the FK problem in a patient-specific setting, we estimate the function $\lambda_0(\boldsymbol{x})$ of Amyloid-$\beta$ protein at the initial time $t=0$. To do that, we project the clinical data derived from PET images with Pittsburgh compound B (PET-PiB) \cite{vanoostveenImagingTechniquesAlzheimer2021}. The PET-PiB adopts a radioligand, which identifies the presence of Amyloid-$\beta$ plaques inside the brain parenchyma (for the specifics about the acquisition techniques of the image used in this work we refer to \cite{OASIS3}). We report the result of the initial concentration rescaled between 0 and 1 and projected on the mesh grid in Figure~\ref{fig:3D:PET}.  In particular, we can observe the presence of large damaged regions $(c\simeq 1)$ in the brainstem and in the thalamus.
\par
\begin{figure}[t!]
    \includegraphics[width=\textwidth]{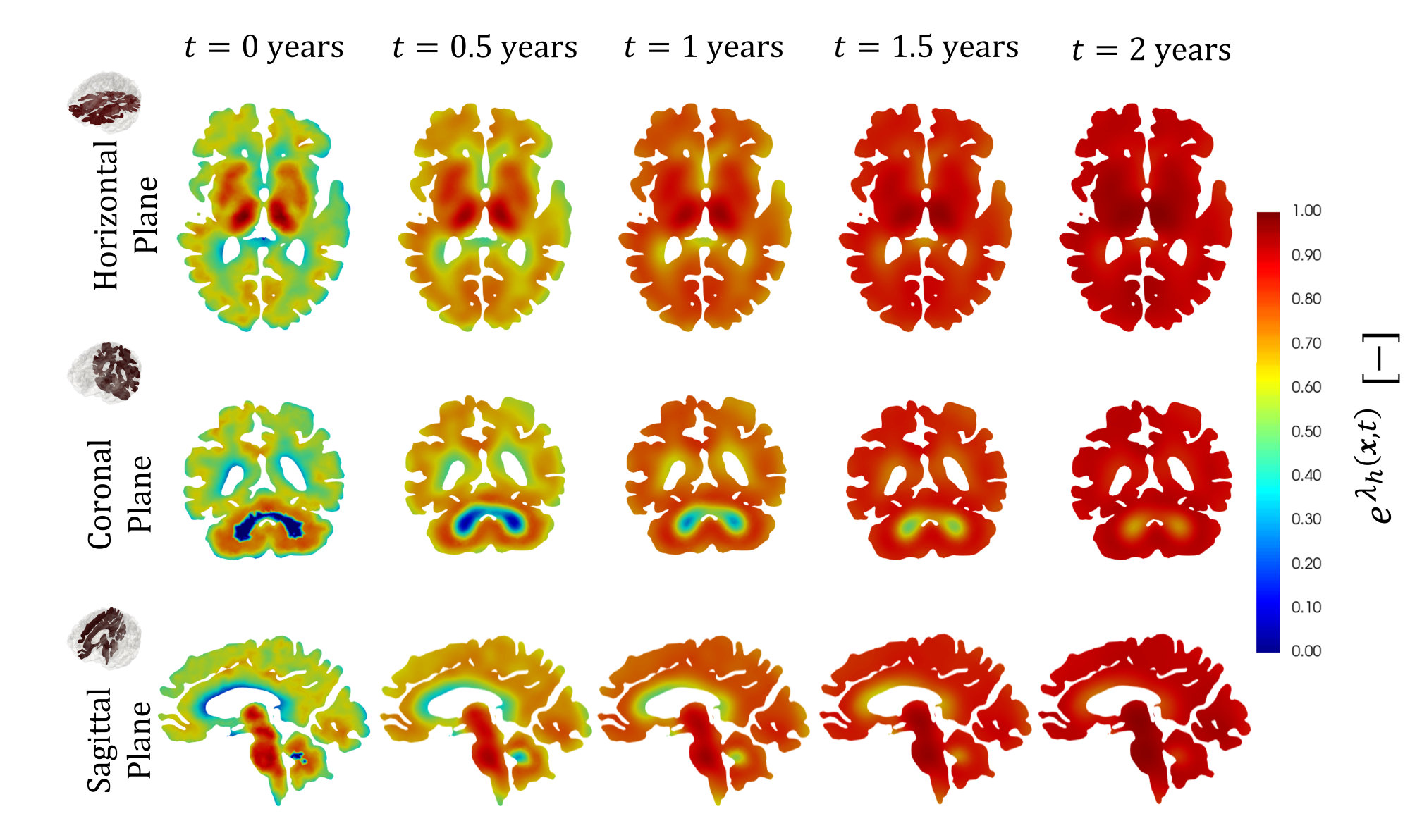}         
    \caption{Test case 4: Patterns of Amyloid-$\beta$ concentration along horizontal (above), coronal (middle), and sagittal (under) planes at different times ($t=0,0.5,1,1.5,2$ years).}
    \label{fig:3DSolution}
\end{figure}
\par
Starting from pathology in an advanced state, we set up a simulation with a final time $T=2$ years and a timestep $\Delta t = 0.01$ years. Concerning the space discretization we adopt the DG method for $p=1$. The nonlinear solver for the resulting system is based on the relaxed Newton method with absolute tolerance equal to $10^{-10}$ and relaxation parameter $\omega=0.75$.
\par
The results are reported in Figure~\ref{fig:3DSolution} for different times ($t=0,0.5,1,1.5,2$ years). The solution is visualized on many slices inside the brain geometry on the three different planes: horizontal, coronal and sagittal. The results show a propagation of the Amyloid-$\beta$ concentration inside the parenchyma, following the typical paths of the pathology \cite{vanoostveenImagingTechniquesAlzheimer2021}. In particular, we can observe a late activation of the cerebellum in the slice along the coronal plane of Figure~\ref{fig:3DSolution} (middle line). This is coherent with Braak's stages of Alzheimer's pathology, which show the presence of Amyloid-$\beta$ accumulation only in the last stages of the pathological development \cite{koychev-amyloid-tau}. Moreover, coherently to the clinical stage of the pathology we are simulating (due to the presence of evident symptoms from the patient's documentation), we can find a generalised misfolding after a few years from the PET acquisition and this is also coherent with what we expected in the disease evolution \cite{vanoostveenImagingTechniquesAlzheimer2021}.

\section{Conclusions}
\label{sec:conclusion}
In this work, we have proposed a positivity-preserving DG method on polygonal and polyhedral grids for the solution of the FK model. The main applicative motivation is the modelling of neurodegeneration caused by the spreading of prionic proteins, such as $\alpha$-synuclein protein in Parkinson's disease and amyloid-$\beta$ in Alzheimer's disease. We have analyzed the existence of the discrete solution by means of the Leray-Schauder theorem and we have discussed the convergence of the numerical scheme.
\par
Numerical tests have been presented both in two and three dimensions. In particular, we have analyzed the convergence in space both with respect to the mesh size and the polynomial order of the method on polygonal grids. Then, we have discussed the convergence in time, by making a comparison between implicit Euler and Cranck-Nicolson schemes. Finally, we have performed a numerical simulation to test the capabilities of the proposed formulation to approximate propagating wavefronts in two dimensions. In this test, we have compared the proposed positivity-preserving method with the polyDG method introduced in \cite{corti:FK}, highlighting the advantages and disadvantages of both formulations.
\par
Finally, we have presented two applications of the proposed scheme in the framework of neurodegenerative diseases. In particuarl, we have performed a simulation of $\alpha$-synuclein spreading on a slice of a real brain in the sagittal plane, constructing a polygonal agglomerated mesh that preserves the quality of both domain boundaries and the interface between white matter and grey matter. Moreover, starting from initial amyloid-$\beta$ concentrations derived from PET images, we have simulated the spreading of amyloid-$\beta$ in a three-dimensional brain in a patient-specific Alzheimer's disease setting. The results obtained in both patient-specific settings are coherent with the clinical literature, showing that the proposed approach is a valuable instrument that can be employed for patient-specific computed-assisted simulations of the evolution of Parkinson's and Alzheimer's neurodegenerative disorders. 
\par
A possible future development of this work consists in extending  the convergence analysis to the general $\vartheta$-method, by proving a discrete entropy decay. Another possibility can be the use of PET images at different times of the disease to calibrate the physical parameters of the Fisher-Kolmogorov model, for example by means of inverse uncertainty quantification methods.

\section*{Acknowledgments}
PFA has been partially funded by the research grants PRIN2017 n. 201744KLJL funded by MUR and PRIN2020 n. 20204LN5N5 funded by MUR. PFA has been partially funded by European Union - Next Generation EU. FB is partially funded by “INdAM - GNCS Project”, codice CUP E53C22001930001. MC, FB and PFA are members of INdAM-GNCS. The brain MRI images were provided by OASIS-3: Longitudinal Multimodal Neuroimaging: Principal Investigators: T. Benzinger, D. Marcus, J. Morris; NIH P30 AG066444, P50 AG00561, P30 NS09857781, P01 AG026276, P01 AG003991, R01 AG043434, UL1 TR000448, R01 EB009352. AV-45 doses were provided by Avid Radiopharmaceuticals, a wholly-owned subsidiary of Eli Lilly.

\section*{Declaration of competing interests}
The authors declare that they have no known competing financial interests or personal relationships that could have appeared to influence the work reported in this article.
\bibliographystyle{hieeetr}
\bibliography{sample.bib}
\end{document}